\documentclass[12pt]{article}

\usepackage{amsmath, amssymb, amsthm, mathrsfs,physics}
\usepackage{hyperref}  
\usepackage{enumitem}  
\usepackage{geometry}  
\newcommand{\diag}[1]{\mathrm{diag}\begin{pmatrix}#1\end{pmatrix}}

\numberwithin{equation}{section}

\newtheorem{theorem}{Theorem}[section]
\newtheorem{lemma}[theorem]{Lemma}
\newtheorem{proposition}[theorem]{Proposition}
\newtheorem{corollary}[theorem]{Corollary}

\theoremstyle{definition}
\newtheorem{definition}[theorem]{Definition}

\theoremstyle{remark}
\newtheorem{remark}[theorem]{Remark}
\newtheorem{example}[theorem]{Example}

\newcommand{\R}{\mathbb{R}}  
\newcommand{\C}{\mathbb{C}}  
\newcommand{\N}{\mathbb{N}}  
\newcommand{\Z}{\mathbb{Z}}  

\allowdisplaybreaks

\title{Frames for source recovery from non-uniform dynamical samples}

\author{Ruchi and Lalit Kumar Vashisht\\
\small{Department of Mathematics, University of Delhi}
}
\date{January 26, 2025}

\begin{document}
	
	\maketitle
	
	\begin{abstract}
Motivated by the work of Aldroubi et al., we investigate the stability of the source term of the discrete dynamical system indexing over a non-uniform discrete set arising from spectral pairs in infinite-dimensional separable Hilbert spaces. Extending results due to Aldroubi et al., firstly, we give a necessary and sufficient condition for the recovery of the source term in finitely many iterations. Afterwards, we derive a necessary condition for the stability of the source term in finitely many iterations when it belongs to the closed subspace of an infinite-dimensional separable Hilbert space. Finally, we give a necessary and sufficient condition for the recovery of the source term in infinitely many iterations. \\
		\textbf{AMS Subject Classification (2020):} 93B30, 42C15,  42C30.\\
		\textbf{Keywords:} Sampling theory; forcing;  frames; reconstruction.
	\end{abstract}
	
	\section{Introduction}
	\label{sec:intro}
	The interplay between frame theory and dynamical sampling has emerged as a riveting area of research in functional analysis, offering profound insights into signal processing, data science, and applied mathematics. Frames were pioneered by Duffin and Schaeffer in a fundamental published paper \cite{NHFS} in 1952 while studying some deep problems in nonharmonic Fourier series. In a nutshell, they endeavored to find condition(s) under which a countable family of complex exponentials in the signal space $L^2(-\gamma, \gamma)$, $\gamma>0$, form a Riesz basis or are complete in $L^2(-\gamma, \gamma)$. The mathematical definition of frames in separable Hilbert spaces is given in Sec. 2 of the paper. Frames were reviewed in a book by Young \cite{Young}. For fundamental properties of frames and their applications in pure mathematics and engineering science, we refer to texts \cite{BHan, C.Heil, KPSA, NPS}.
	
In one of the directions of applications of frames in engineering science, dynamical sampling is a relatively recent framework that addresses the challenge of reconstructing signals from spatially and temporally evolving measurements. The origins of dynamical sampling trace back to the seminal work of Aldroubi et al. \cite{DSTS,ERSS} in 2013, where they investigated conditions under which a function can be reconstructed from coarse samples of its iterates under a linear operator. Recent research has also accentuated applying dynamical sampling to new challenges, such as real-time data streaming, dynamic environments, and non-stationary signals. These advancements have permeated the applicability of dynamical sampling techniques in areas ranging from medical diagnostics and environmental monitoring to robotics and artificial intelligence. Aldroubi et al. studied applications of Bessel sequences, frames, and Riesz bases in sampling theory in a series of papers, see \cite{DSIS,  DSTS, DSSIO, DS} and many references therein.
	
On the other hand,  Fuglede \cite{CSPDG}  in the study of commuting self-adjoint partial differential operators introduced the concept of spectral pairs.
By using the theory of spectral pairs in the Lebesgue space  $L^2(\mathbb{R})$,  Gabardo and Nashed \cite{NUMRA} generalized the notion of multiresolution analysis (MRA). They introduced the concept of non-uniform multiresolution analysis. In non-uniform multiresolution analysis, the corresponding translation set may not be a group but is  a spectrum associated with one-dimensional spectral pairs.
It is applicable in input/output models based on shifts of data points, where shifts may not be uniform in nature. Frame conditions for the discrete Gabor system related to spectral pairs can be found in \cite{DVNGF}.

	\subsection{Related work and contribution}
	Dynamical sampling problems can be stratified in three types depending on the specific variable of interest. When the objective is to determine the conditions on spatial sampling set for recovering initial state, the problem is referred to as the \textbf{space-time trade-off} \cite{DSTS,ICNO,DS,SFBF,DSFIS,CDDS}. On the other hand, when the aim is to determine the conditions for the recovery of dynamic operator $A$, the problem is categorized as \textbf{system identification} \cite{KSMDS,DSARN,ESADS,SIDS} and the problem of identifying specific types of source terms that drive the dynamics of the system is known as \textbf{source recovery problems}, see \cite{RRDSE,PADBF,DSRSCS,CFSRDS} and many  references therein.
	
		The formal development of dynamical sampling as a distinct area began in the early 2010s. In 2015, Aldroubi et al. introduced the concept of sampling in evolving systems in their landmark work in \cite{ERSS}. They studied sampling and reconstruction for functions evolving under the action of linear operators, laying the groundwork for future studies.  Aldroubi, Cabrelli, and Molter \cite{DS}  developed the initial framework for finite-dimensional spaces and certain classes of self-adjoint operators in infinite-dimensional spaces, providing the theoretical foundation for subsequent studies. In \cite{DSFIS}, Cabrelli, Molter and Petrosyan extended this work to investigate dynamical sampling on finite index sets. They provided a theoretical solution to the problem where a signal evolving over time through iterates of an operator can be recovered from its space-time samples. By 2021, the focus shifted towards the reconstruction of bandlimited functions from space–time samples. Researchers analyzed the problem of reconstructing a bandlimited function from the space–time samples of its states resulting from convolution with a kernel, addressing challenges in stable reconstruction \cite{SFBF}. In recent work,   Huang, Neuman, Tang, and Xie explored convolutional dynamical sampling, focusing on the recovery of bandlimited heat diffusion fields from space-time samples in both finite and infinite-dimensional cases. Their work extended previous findings and emphasized the density of space-time sampling sets \cite{CDSR}.  For applications of sampling in sensor networks, it is necessary to mention the published papers  \cite{SSRS, SRDL}.

 Most recently,  Aldroubi et al. \cite{CFSRDS} advanced the field by characterizing frames for stable recovery of source terms in dynamical systems. 	Specifically, they consider the discrete dynamical system of the form $$x_{n+1}=Ax_{n}+w, n\in \N,w \in W,$$ where $x_{n} \in \mathcal{H}$ is the $n$-th state of the system, and $\mathcal{H}$ is a separable Hilbert space. The operator $A \in \mathcal{B}(\mathcal{H}), x_{0} \in \mathcal{H}$ is the initial state, and $W$ is a closed subspace of $\mathcal{H}$. Time-space sample measurements
\begin{align*}
D(x_{0},w)=[\langle x_{n},g_{j} \rangle]_{n \geq 0,\,j \geq 1}
\end{align*}
	are obtained by inner products $\langle x_{n},g_{j} \rangle$ with vectors of a Bessel system $\{g_{j}\}_{j \geq 1} \subset \mathcal{H}$.
	It is given in \cite{CFSRDS} that frame conditions related to the above dynamical system can be used in environmental monitoring applications for identifying the locations and magnitude of pollution sources. Aldroubi et al., in \cite{CFSRDS}, focused on the stable recovery of a source term by using time-space sample measurements formed by inner products with vectors from a given Bessel system in separable Hilbert spaces.  They proved necessary and sufficient conditions for the stable recovery of constant source terms from time-space samples in separable Hilbert spaces.

On the other hand, Gabardo and Nashed, in \cite{NUMRA}, generalized the concept of multiresolution analysis (MRA) by using the theory of spectral pairs in the Lebesgue space  $L^2(\mathbb{R})$. The theory of spectral pairs and non-uniform wavelets is instrumental in developing advanced numerical methods for solving partial differential equations, particularly in cases where traditional grid-based techniques fail or are computationally expensive. In this framework, the translation set is not necessarily a group but is instead a spectrum associated with one-dimensional spectral pairs. They provided a characterization for the existence of non-uniform wavelets associated with non-uniform multiresolution analysis in the  space $L^2(\mathbb{R})$. In \cite{JN}, they established an analogue to Cohen’s condition for non-uniform multiresolution analysis, highlighting that, unlike in traditional MRA, non-uniform multiresolution analysis may not always produce the associated wavelets. Recently, in \cite{DVNGF}, two  authors of this paper studied
non-uniform frames with Gabor structure  in the discrete signal space $\ell^2(\Lambda)$, where $\Lambda$ is a non-uniform lattice (see below). Recall that some  real life problems, for example, sampling,  data analysis,  signal analysis, etc., are based on  discrete data points, where the indexing set may not be a uniform lattice. In these directions, wavelets and frames over non-uniform indexing play a significant role \cite{DVNGF}. For applications of discrete frames with different structure, we refer to \cite{BHan} and many references therein.

Motivated by the above work, we study  stability of the source term of the non-uniform discrete dynamical system in infinite dimensional separable Hilbert spaces. More precisely, we deliberate on indexing the dynamical system and sampling vectors by sets arising from spectral pairs, which is not necessary a group but a spectrum which is based on the theory of spectral pairs  \cite{CSPDG, NUMRA}.  Extending results due to Aldroubi et al. \cite{CFSRDS}, we give a  necessary and sufficient condition for the recovery of the source term in finitely many iterations. Afterwards, we derive a necessary condition for the stability of the source term in finitely many iterations when it belongs to the closed subspace of a infinite dimensional separable  Hilbert space. A  necessary and sufficient condition for the recovery of the source term in infinitely many iterations is also given.
	
	\subsection{Structure of the paper} To make  the paper  self-contained, Section \ref{Section 2} gives basic notations and foundational results relevant to Bessel sequences, discrete Hilbert frames in separable Hilbert spaces and spectral pairs. In Section \ref{Section 3}, we present the main results of the paper. First, we fabricate the general setup for the non-uniform discrete dynamical system. Theorem \ref{Th3.8} highlights a fundamental property of the space $\mathcal{B}^{s}(\ell^2(\Lambda),\ell^\infty(\Lambda)) $ and its pivotal role in ensuring stability during the reconstruction process. Theorem \ref{Th3.12} provides a characterization of the stable recovery of the source term in finitely many iterations. In Theorem \ref{Th3.14}, we give a necessary condition for the stable recovery of the source term when it belongs to the closed subspace of infinite-dimensional separable Hilbert space. A characterization of the stable recovery of the source term in infinitely many iterations is given in Theorem \ref{Th3.17}. As a fruitage of Theorem \ref{Th3.14}, for the characterization of stable reconstruction of the non-uniform discrete dynamical system \eqref{Dynamical System} generated by an operator $A$ with spectral radius $\rho(A)<1$,  see Theorem \ref{Th3.19}. Illustrative examples a are also given to support our results.

	\section{Preliminaries}\label{Section 2}
	Throughout the paper, $\mathbb{Z}$ $\mathbb{N}$, $\mathbb{R}$ and $\mathbb{C}$ denote the set of all integers, natural numbers, real numbers and complex numbers, respectively. Symbol $\mathcal{H}$ denotes an infinite-dimensional separable (real or complex) Hilbert space.
	
	\subsection{Discrete Hilbert frames}
	A countable collection of vectors  $\{f_{k}\}_{k \in \mathbf{I}}$ in
	$\mathcal{H}$  is called a  \textbf{discrete frame} (or discrete Hilbert frame) for $\mathcal{H}$ if there exist positive real scalars $ \alpha_{0}$ and $\beta_{0}$ such that
	\begin{align}
		\alpha_{0}\norm{f}^{2} \leq \sum_{k \in \mathbf{I}}\abs{\langle f, f_{k} \rangle}^{2} \leq \beta_{0}\norm{f}^{2} \,\,\, \text{for all}\,\, f \in \mathcal{H}.
		\label{eq:2.1}
	\end{align}
	The constants $\alpha_{0}$ and $\beta_{0}$ are called \textbf{lower} and \textbf{upper frame bounds} of $\{f_{k}\}_{k \in \mathbf{I}}$, respectively. If $\alpha_{0}=\beta_{0}$, then $\{f_{k}\}_{k \in \mathbf{I}}$ is called tight frame and if $\alpha_{0}=\beta_{0}=1$, then it is called the \textbf{Parseval frame} for $\mathcal{H}$. If $\{f_{k}\}_{k \in \mathbf{I}}$ only satisfies the upper inequality in \eqref{eq:2.1}, then we say $\{f_{k}\}_{k \in \mathbf{I}}$ is a \textbf{Bessel sequence} in $\mathcal{H}$ with \textbf{Bessel bound} $\beta_{0}$.
	Associated with the Bessel sequence $\{f_{k}\}_{k \in \mathbf{I}}$, the bounded linear operator $T:\ell^{2}(\N) \rightarrow \mathcal{H}$ given by
\begin{align*}
T\{c_{k}\}_{k \in \mathbf{I}}=\sum_{k \in \mathbf{I}}c_{k}f_{k}
\end{align*}
is called the \textbf{synthesis operator} or the \textbf{pre-frame operator} and the adjoint operator of $T$ is the map $T^*:\mathcal{H} \rightarrow \ell^{2}(\N)$ given by
\begin{align*}
 T^*(f)=\{\langle f,f_{k} \rangle\}_{k \in \mathbf{I}},
 \end{align*}
and known as  the \textbf{analysis operator}. The \textbf{frame operator} is the map  $\Theta = TT^*:\mathcal{H} \rightarrow \mathcal{H}$ given by
\begin{align*}
\Theta f=\sum_{k \in \mathbf{I}}\langle f,f_{k} \rangle f_{k}.
\end{align*}
If  $\{f_{k}\}_{k \in \mathbf{I}}$ is a frame for $\mathcal{H}$,  then the frame operator $\Theta$ is bounded, linear and invertible on $\mathcal{H}$.

The following lemma will be used in the sequel.
	
	\begin{lemma}\cite[p. 221]{C.Heil}\label{L:2.1}
		Let $\{f_{k}\}_{k \in \mathbf{I}}$ be a frame  for $\mathcal{H}$ and let $f \in \mathcal{H}$. If $f$ has a representation $f=\sum\limits_{k \in \mathbf{I}}c_{k}f_{k}$ for some coefficients $\{c_{k}\}_{k \in \mathbf{I}}$, then
		\begin{align*}
			\sum_{k \in \mathbf{I}}{\abs{c_{k}}}^{2}=\sum_{k \in \mathbf{I}}{\abs{\langle f,S^{-1}f_{k}\rangle}}^{2} + \sum_{k \in \mathbf{I}}{\abs{c_{k}-\langle f,S^{-1}f_{k}\rangle}}^{2}
		\end{align*}
	\end{lemma}
Let  $\{f_{k}\}_{k \in \mathbf{I}}$ be a frame for $\mathcal{H}$ with frame operator $\Theta$.  Then, the family  $\{\Theta^{-1}f_{k}\}_{k \in \mathbf{I}}$ is a frame for $\mathcal{H}$. It is called the \textbf{canonical dual} of $\{f_{k}\}_{k \in \mathbf{I}}$. Instinctive search for the freedom of various choices of coefficients instead of the conventional frame coefficients $ \{ \langle f,S^{-1}f_k \rangle \}_{k \in \mathbf{I}}$, led to the theory of dual frames, which is quite important in applied mathematics.
Let $\{f_{k}\}_{k \in \mathbf{I}}$ be a frame for $\mathcal{H}$. Then, the sequence $\{g_{k}\}_{k \in \mathbf{I}}$ is called a \textbf{dual frame} of $\{f_{k}\}_{k \in \mathbf{I}}$ if $\{g_{k}\}_{k \in \mathbf{I}}$ is a frame for $\mathcal{H}$ and each $f \in \mathcal{H}$ can be expressed as
	\begin{align}
		f = \sum_{k \in \mathbf{I}}\langle f,g_{k} \rangle f_{k}.
		\label{eq:2.2}
	\end{align}
 That is, every frame $\{f_{k}\}_{k \in \mathbf{I}}$ admits at least one dual frame, namely  $\{\Theta^{-1}f_{k}\}_{k \in \mathbf{I}}$, the canonical dual frame.  A dual
	frame of $\{f_{k}\}_{k \in \mathbf{I}}$ other than $\{\Theta^{-1}f_{k}\}_{k \in \mathbf{I}}$ is known as \textbf{alternative dual}. For fundamental properties of dual frames in separable Hilbert spaces, we refer to \cite{C.Heil}.

	\subsection{Non-uniform frames}	
We use the symbol $\Lambda$ to denote the set $\Lambda:=\left\{0,\frac{r}{N}\right\} + 2\Z$, where $N \in \N$, $r$ be a fixed odd integer coprime with $N$ such that $1 \leq r \leq 2N-1$. Note that  $\Lambda$, in general, not  necessarily a group but is a spectrum associated with a one-dimensional spectral pair.  Gabardo and Nashed \cite{NUMRA} defined the notion of  spectral pair in their study of non-uniform wavelets in the space $L^2(\R)$.
	\begin{definition}\cite{NUMRA}
		Let $\Omega \subset \R$ be measurable and $\Lambda \subset \R$ a countable subset. If the
		collection $\{|\Omega|^{-\frac{1}{2}}e^{2\pi i\lambda.}\chi_{\Omega}(.)\}_{\lambda \in \Lambda}$ forms complete orthonormal system for $L^2(\Omega)$, where
		$\chi_{\Omega}$ is indicator function on $\Omega$ and $|\Omega|$ is Lebesgue measure of $\Omega$, then the pair $(\Omega,\Lambda)$ is a spectral pair.
	\end{definition}
	
	\begin{example}\cite{NUMRA}
		Let  $N \in \N$, $r$ be a fixed odd integer coprime with $N$ such that \break $1 \leq r \leq 2N-1$, and let $\Lambda=\left\{0,\frac{r}{N}\right\} + 2\Z$, and $\Omega=[0, \frac{1}{2})  \cup  \left[\frac{N}{2}, \frac{N+1}{2}\right)$. Then, $(\Omega,\Lambda)$ is a spectral pair.
		\end{example}
\begin{definition}
A frame of the form  $\{f_{k}\}_{k \in \Lambda} \subset \mathcal{H}$  for $\mathcal{H}$ is called a \textbf{non-uniform frame}  for $\mathcal{H}$.
\end{definition}
For non-uniform frames with discrete Gabor and wavelet structure, we refer to  \cite{HK20, HK21, DVNGF} and references therein.

	\section{Non-uniform discrete dynamical system}\label{Section 3}
We begin this section with the definition of non-uniform discrete dynamical system. Let $N \in \N$ and $r$ be a fixed odd integer co-prime with $N$ such that $1 \leq r \leq 2N-1$. We recall a notation $\Lambda:=\left\{0,\frac{r}{N}\right\} + 2\Z$.  Until and unless specified, symbol $[2K]$  denotes the set
\begin{align*}
[2K]:=\Big\{-2K,-2K+\frac{r}{N},\ldots,-2, -2+\frac{r}{N},0,\frac{r}{N},2,\ldots,2K-2,2K-2+\frac{r}{N}\Big\}, K \in \N;
 \end{align*}
 and $|[2K]|$ denotes the cardinality of the  set $[2K]$. The set $[2K]$ is only used to represent finite number of iterations.
	
\begin{definition}
Let $A$ be a bounded linear operator on  $\mathcal{H}$,  $W$ be a closed subspace of $\mathcal{H}$ and  $w \in W$  is the source term or forcing term.	A system of the form
	\begin{align}
		x_{\lambda+\frac{r}{N}} = &Ax_{\lambda} +w, \,  \lambda \in 2\Z; \notag\\
		x_{\lambda+2-\frac{r}{N}} =& Ax_{\lambda} +w, \, \lambda \in 2\Z^{+}+\frac{r}{N} \cup \left\{\frac{r}{N}\right\}; \notag\\
		x_{\lambda-2-\frac{r}{N}} = & Ax_{\lambda} +w, \, \lambda \in 2\Z^{-}+\frac{r}{N},
		\label{Dynamical System}
	\end{align}
is called a \textbf{non-uniform discrete dynamical system}.	Vectors $x_{\lambda} \in \mathcal{H}$,  $\lambda \in \Lambda$,  is the \textbf{$\lambda$-th state} of the non-uniform discrete dynamical system in \eqref{Dynamical System}. The terms $x_{0}$ and  $x_{-2}$ are called \textbf{initial states}.
\end{definition}

	Time-space sample measurements
	\begin{align}\label{eq:1.2}
		D(x_{0},x_{-2},w)=& \left[\langle x_{\lambda},g_{\lambda'}\rangle\right]_{\lambda,\,      \lambda' \in \Lambda},	
	\end{align}
	where $\Lambda=\left\{0,\frac{r}{N}\right\} + 2\Z,N \geq 1$ is an integer, and $r$ be a fixed odd integer coprime with $N$ such that $1 \leq r \leq 2N-1$, are obtained via inner products $\langle x_{\lambda},g_{\lambda'}\rangle$ with the vectors of a Bessel system $\{g_j\}_{j \in \Lambda} \subset \mathcal{H}$, referred to as the set of \textbf{spatial sampling vectors}. These measurements are organized in the matrix $D(x_{0},x_{-2},w)$, which is known as the \textbf{data matrix}. This matrix is also referred to as the data of the system, or alternatively as the set of time-space samples, measurements, or observations.

	
	To define the notion of stable reconstruction, we need to specify the measurement spaces, where the data resides, along with an appropriate norm. This framework enables us to represent the reconstruction operator
	$\mathscr{R}$ as a bounded linear mapping from the data space $\mathcal{B}$ to the Hilbert space $\mathcal{H}$. The following spaces will be used in the sequel:
		\begin{itemize}
			\item The sequence space $\ell^2(\Lambda):=\Big\{x=\{x_{\lambda}\}_{\lambda \in \Lambda} \subset \C : \sum_{\lambda \in \lambda}{\abs{x_{\lambda}}}^{2} < \infty\Big\}$ is a Hilbert space with respect to the inner product defined by
$\langle x,y\rangle = \sum_{\lambda \in \Lambda}x_{\lambda}\bar{y_{\lambda}},\,\, x = \{x_{\lambda}\}_{\lambda \in \Lambda}, \, y=\{y_{\lambda}\}_{\lambda \in \Lambda} \in \ell^2(\Lambda)$.
	
			\item The sequence space $\ell^\infty(\Lambda):=\Big\{x=\{x_{\lambda}\}_{\lambda \in \Lambda} \subset \C : \sup_{\lambda \in \lambda}\abs{x_{\lambda}} < \infty\Big\}$ is a Banach space endowed with the norm ${\norm{x}}_{\ell^\infty(\Lambda)}=\sup\limits_{\lambda \in \Lambda}\abs{x_{\lambda}}$.
			\item We write  $\C^{[2K]}:= \Big\{x={(x_{\lambda})}_{\lambda \in [2K]} : x_{\lambda} \in \C\Big\}$.
		\end{itemize}
		
		Now, we familiarize the spaces $\mathcal{B}(\ell^2(\Lambda),\C^{[2K]})$, $\mathcal{B}(\ell^2(\Lambda),\ell^\infty(\Lambda))$, and $\mathcal{B}^{s}(\ell^2(\Lambda),\ell^\infty(\Lambda))$ which are vital for our work.
		
			\begin{definition}
    The space $\mathcal{B}(\ell^2(\Lambda),\C^{[2K]})$ is the set of all infinite matrices $T=[a_{ij}]_{i \in [2K],j \in \Lambda}$ such that each row
   $r_{i}$ of $T$ belongs to $\ell^2(\Lambda)$. We endow $\mathcal{B}(\ell^2(\Lambda),\C^{[2K]})$ with the norm
   $ \norm{T}_{\ell^2(\Lambda) \rightarrow \C^{[2K]}}=\sum_{i \in [2K]}{\norm{r_{i}}}_{\ell^2(\Lambda)}$.
			\end{definition}
			The space $\mathcal{B}(\ell^2(\Lambda),\C^{[2K]})$ is a Banach space which is tantamount to the space of bounded linear operators from $\ell^2(\Lambda)$  to $\C^{[2K]}$, endowed with the operator norm.
		
		\begin{definition}
			The space $\mathcal{B}(\ell^2(\Lambda),\ell^\infty(\Lambda))$ is the set of all doubly infinite matrices $T=[a_{ij}]_{i,j \in \Lambda}$ such that each row $r_{i}$ of $T$ belongs to $\ell^2(\Lambda)$, and $\sup_{i \in \Lambda}\norm{r_{i}}_{\ell^2(\Lambda)}$ is finite.	We endow $\mathcal{B}(\ell^2(\Lambda),\ell^\infty(\Lambda))$ with the norm $\norm{T}_{\ell^2(\Lambda) \rightarrow \ell^{\infty} (\Lambda)}=\sup_{i \in \Lambda}{\norm{r_{i}}}_{\ell^2(\Lambda)}$.
		\end{definition}

		The space $\mathcal{B}(\ell^2(\Lambda),\ell^\infty(\Lambda))$ is a Banach space which is tantamount to the space of bounded linear operators from $\ell^2(\Lambda)$  to $\ell^\infty(\Lambda)$, endowed with the operator norm. Now, we are ready to define the subspace $\mathcal{B}^{s}(\ell^2(\Lambda),\ell^\infty(\Lambda))$.

	\begin{definition}
		
	The space $\mathcal{B}^{s}(\ell^2(\Lambda),\ell^\infty(\Lambda))$ is the set of matrices $\{T=[a_{ij}]:i,j \in \Lambda\} \subset \mathcal{B}(\ell^2(\Lambda),\ell^\infty(\Lambda)) $ such that there exists a $z \in \ell^2(\Lambda)$ satisfying
	$\lim_{\abs{i}\rightarrow\infty} {\norm{r_{i}-z}}_{\ell^2(\Lambda)}=0$. We endow $\mathcal{B}^{s}(\ell^2(\Lambda),\ell^\infty(\Lambda))$ with the norm induced by $\mathcal{B}(\ell^2(\Lambda),\ell^\infty(\Lambda))$.
	\end{definition}
	We ruminate for the following  two cases of non-uniform discrete dynamical system:
	\begin{enumerate}
		\renewcommand{\labelenumi}{\roman{enumi}.}
		\item In the first case, the data matrix 	 $D(x_{0},x_{-2},w)= \left[\langle x_{\lambda},g_{\lambda'}\rangle\right]_{\lambda,\, \lambda' \in \Lambda} $
		is obtained from
		finitely many iterations $|[2K]|$.
		\item In the second case, the data matrix 	 $D(x_{0},x_{-2},w)= \left[\langle x_{\lambda},g_{\lambda'}\rangle\right]_{\lambda,\, \lambda' \in \Lambda} $
		is obtained from
		infinitely many time iterations.
	\end{enumerate}
	
	For the first case, all data measurements are carried out in the space $\mathcal{B}(\ell^2(\Lambda),\C^{[2K]})$. In the second case of infinitely many time iterations, we  utilize the space $\mathcal{B}^{s}(\ell^2(\Lambda),\ell^\infty(\Lambda))$ which is a closed subspace of $\mathcal{B}(\ell^2(\Lambda),\ell^\infty(\Lambda))$.
	
	\subsection{Generalized non-uniform discrete dynamical system}
	In this subsection, we deal with the non-uniform discrete  dynamical systems that are the generalized version of \eqref{Dynamical System}. In this general setting, we assume that the states $x_{\lambda},\lambda \in \Lambda$ are obtained via the recursive relation
	\begin{align}
			x_{\lambda}=&
			\begin{cases}
			\mathscr{F}_{\lambda}(x_{0},x_{\frac{r}{N}},x_{2},x_{2+\frac{r}{N}} \cdots,x_{\lambda-2+\frac{r}{N}},w),\,\lambda \in  2\Z^{+};&\\
				\mathscr{F}_{\lambda}(x_{0},x_{\frac{r}{N}},x_{2},x_{2+\frac{r}{N}} \cdots,x_{\lambda-\frac{r}{N}},w),\,\lambda \in  2\Z^{+}+\frac{r}{N} \cup \left\{\frac{r}{N}\right\};&\\
				\mathscr{F}_{\lambda}(x_{-2},x_{-2+\frac{r}{N}},x_{-4},x_{-4+\frac{r}{N}} \cdots,x_{\lambda+2+\frac{r}{N}},w),\,\lambda \in  2\Z^{-};&\\
				\mathscr{F}_{\lambda}(x_{-2},x_{-2+\frac{r}{N}},x_{-4},x_{-4+\frac{r}{N}} \cdots,x_{\lambda-\frac{r}{N}},w),\,\lambda \in  2\Z^{-}+\frac{r}{N},
		\end{cases}
		\label{Generalized Dynamical System}
		\end{align}
		with $w$ belongings to the closed subspace $W$ of $\mathcal{H}$. In particular, $\mathscr{F}_{\lambda}$ can be a non-linear functional of its arguments.

		To present certain results in the context of setting \eqref{Generalized Dynamical System}, we assume that the system satisfies the following properties:
		\begin{enumerate}
			\renewcommand{\labelenumi}{\roman{enumi}.}
  			\item For each $w \in W$, there is a corresponding unique pair of stationary states. More explicitly, given any $w \in W$, there is a pair of initial states $(x_{0}(w),x_{-2}(w))$ such that
  			$$x_{\lambda}=\frac{(x_{0}+x_{-2})(w)}{2} \,\, \text{for all}\, \lambda \in \Lambda.$$
  			\item The correspondence between $w$ and its unique pair of stationary states  $(x_{0}(w),x_{-2}(w))$  is bounded. That is, the mapping $S:W \longrightarrow \mathcal{H}$ defined by
  			$$S(w)=\frac{(x_{0}+x_{-2})(w)}{2}$$
  			is a bounded linear operator, and $S$ is called as the stationary mapping operator.
  			\item For any source term $w \in W $ and any arbitrary initial states $x_{0},x_{-2} \in \mathcal{H}$, we have
  			$$ \lim_{\abs{\lambda}\rightarrow\infty} x_{\lambda}=S(w),$$
  			where the above limit is in $\norm{.}_{\mathcal{H}}$.
  		
		\end{enumerate}
		
		\begin{definition}\label{Def:2.3}
			A non-uniform discrete dynamical system \eqref{Generalized Dynamical System} satisfying the above properties $(i)-(iii)$ is denoted by the quadruple $(\mathcal{H},W,\mathscr{F},S)$.
		\end{definition}
		
		\subsection{Stable recovery for non-uniform discrete dynamical system}
		In this subsection, we define the notion of stable recovery for non-uniform discrete dynamical system and presents some important properties of the space $\mathcal{B}^{s}(\ell^2(\Lambda),\ell^\infty(\Lambda))$.
		
		 We start with the non-uniform discrete dynamical system of the form \eqref{Dynamical System} or of the form
		\eqref{Generalized Dynamical System}, with arbitrary initial states $x_{0}, x_{-2} \in \mathcal{H}$ and measurements $D(x_{0},x_{-2},w)$
		given by sampling through a Bessel sequence $\{g_j\}_{j \in \Lambda}$ in $\mathcal{H}$ as in \eqref{eq:1.2}, then
		\begin{enumerate}
			\renewcommand{\labelenumi}{\roman{enumi}.}
			\item the source term $w \in W \subseteq \mathcal{H}$ is said to be stably recovered from the data matrix $D(x_{0},x_{-2},w)$ in finitely many time iterations if there exists a bounded linear operator $\mathscr{R}:\mathcal{B}(\ell^2(\Lambda),\C^{[2K]}) \longrightarrow \mathcal{H}$ such that
					$\mathscr{R}(D(x_{0},x_{-2},w))=w \; \text{for all} \; x_{0},x_{-2} \in \mathcal{H}, w \in W $.

			\item the source term $w \in W \subseteq \mathcal{H}$ is said to be stably recovered from the data matrix $D(x_{0},x_{-2},w)$ in infinitely many time iterations if there exists a bounded
			linear operator 	$\mathscr{R}:\mathcal{B}^{s}(\ell^2(\Lambda),\ell^\infty(\Lambda)) \longrightarrow \mathcal{H}$
			such that
			$\mathscr{R}(D(x_{0},x_{-2},w))=w \; \text{for all}\; x_{0}, x_{-2} \in \mathcal{H}, w \in W $.
			\end{enumerate}
	
	The following Lemma is an adaptation of \cite[Lemma 4.1]{CFSRDS}.
   		\begin{lemma}\label{L3.4}
		Let $T:\ell^2(\Lambda) \rightarrow \ell^\infty(\Lambda)$ be a bounded linear operator. Let $[a_{ij}]$ be the matrix representation of $T$ with respect to the canonical basis $(e_{j})_{j \in \Lambda}$, i.e.,
        $ Te_{j}=\sum_{i \in \Lambda}a_{ij}e_{j}, \, j \in \Lambda$.
		Then,
        \begin{align*}
        \norm{T}_{\ell^2(\Lambda) \rightarrow \ell^{\infty} (\Lambda)}=\sup_{i \in \Lambda}{\Big(\sum_{j \in \Lambda}{\abs{a_{ij}}}^{2}\Big)}^{\frac{1}{2}}.
 \end{align*}
	\end{lemma}

		\begin{remark}\label{R3.5}
		If 	$\{g_j\}_{j \in \Lambda} \subset \mathcal{H} $ is a Bessel sequence, then $\mathcal{B}^{s}(\ell^2(\Lambda),\ell^\infty(\Lambda))$ (where $s$ stands for strong)
		consists of all operators $T \in \mathcal{B}(\ell^2(\Lambda),\ell^\infty(\Lambda))$ such that the limit
		\begin{align}
			\lim_{\abs{i}\rightarrow\infty} \sum_{j \in \Lambda}a_{ij}g_{j}
			\label{eq:3.4}
		\end{align}
		exists in the norm of $\mathcal{H}$. Since all the separable Hilbert spaces are unitarily
		equivalent, so the definition does not depend on $\mathcal{H}$. Moreover, we write $\lim{T\{g_{j}\}_{j \in \Lambda}}$ for the limit in \eqref{eq:3.4}.
	\end{remark}
	
	The next result is an adaptation of \cite[Lemma 4.4]{CFSRDS}. For the sake of completeness, we provide its proof.
	\begin{lemma}\label{L3.6}
		Let $T=[a_{ij}] \in \mathcal{B}(\ell^2(\Lambda),\ell^\infty(\Lambda))$. Then $T \in \mathcal{B}^{s}(\ell^2(\Lambda),\ell^\infty(\Lambda))$ if and only if the rows of $T$ are norm convergent in $\ell^2(\Lambda),i.e.,$
		there is a vector $z \in \ell^2(\Lambda) $ such that
		\begin{align*}
			 \lim_{\abs{i}\rightarrow\infty} \norm{r_{i}-z}_{\ell^2(\Lambda)}=0,
		\end{align*}
			where $ r_{i}=(\ldots,a_{i-2},a_{i-2+\frac{r}{N}},a_{i0},a_{i\frac{r}{N}},\ldots) $ denotes the $i$-th  row of $T$.
	\end{lemma}
	
	\begin{proof}
 First, suppose that $T \in \mathcal{B}^{s}(\ell^2(\Lambda),\ell^\infty(\Lambda))$.
		Then, by Remark \ref{R3.5},
		$T \in \mathcal{B}(\ell^2(\Lambda),\ell^\infty(\Lambda))$ such that the limit
		$$ \lim_{\abs{i}\rightarrow\infty} \sum_{j \in \Lambda}a_{ij}g_{j}$$
		exists in the norm of $\mathcal{H}$ for any Bessel sequence
		$\{g_j\}_{j \in \Lambda} \subset \mathcal{H} $
		and since all the orthonormal bases are Bessel sequences, so if we take  the  Bessel sequence which is an orthonormal basis in $\mathcal{H}$, then the series in \eqref{eq:3.4} is unitarily equivalent to the vector
		\begin{align*}
			z =(\ldots,z_{-4},z_{-4+\frac{r}{N}},z_{-2},z_{-2+\frac{r}{N}},z_{0},z_{\frac{r}{N}},z_{2},z_{2+\frac{r}{N}},\ldots) \in \ell^2(\Lambda).
		\end{align*}
		Thus, the existence of limit is equivalent to the rows are convergent in $\ell^2(\Lambda)$.
		
		For the converse part, let us assume that $T \in \mathcal{B}(\ell^2(\Lambda),\ell^\infty(\Lambda))$ and its rows are convergent in $\ell^2(\Lambda)$ to a vector
		$z =(\ldots,z_{-4},z_{-4+\frac{r}{N}},z_{-2},z_{-2+\frac{r}{N}},z_{0},z_{\frac{r}{N}},z_{2},z_{2+\frac{r}{N}},\ldots) \in \ell^2(\Lambda) $, and fix an arbitrary Bessel sequence $\{g_j\}_{j \in \Lambda} \subset \mathcal{H} $ with Bessel bound  $\beta>0$.
Then,
\begin{align*}
	h:=\sum_{j \in \Lambda}z_{j}g_{j}
\end{align*}
is a well-defined element of $\mathcal{H}$.
		Moreover,
		\begin{align*}
			\Big\|h-\sum_{j \in \Lambda}a_{ij}g_{j}\Big\|_{\mathcal{H}}=\Big\|\sum_{j \in \Lambda}z_{j}g_{j}-\sum_{j \in \Lambda}a_{ij}g_{j}\big\|_{\mathcal{H}}
			=\Big\|\sum_{j \in \Lambda}(z_{j}-a_{ij})g_{j}\Big\|_{\mathcal{H}}
			= \sqrt{\beta}\norm{z-r_{i}}_{\ell^2(\Lambda)}.
		\end{align*}
		By assumption,
			$	\norm{z-r_{i}}_{\ell^2(\Lambda)} \rightarrow 0 \;  \text{as} \; \abs{i}\rightarrow\infty $.
    Hence,
		$\lim_{\abs{i}\rightarrow\infty} \sum_{j \in \Lambda}a_{ij}g_{j} = h $.
	This implies that
		$T \in \mathcal{B}^{s}(\ell^2(\Lambda),\ell^\infty(\Lambda))$.
	\end{proof}
	
	The Lemma \ref{L3.6} gives the following corollary:
	\begin{corollary}\label{C3.7}
		$\mathcal{B}^{s}(\ell^2(\Lambda),\ell^\infty(\Lambda))$ is a closed subspace of $\mathcal{B}(\ell^2(\Lambda),\ell^\infty(\Lambda))$.
	\end{corollary}
	
	\begin{proof}
		
		The proof follows along the same lines as presented in \cite[Corollary 4.5]{CFSRDS}.
	\end{proof}

	Propelled by the work of \cite[Theorem 3.1]{CFSRDS}, the next result ratifies an important property of the space $\mathcal{B}^{s}(\ell^2(\Lambda),\ell^\infty(\Lambda))$, that it is a natural domain of the reconstruction operator $\mathscr{R}$ and the operator $\mathscr{R}:\mathcal{B}^{s}(\ell^2(\Lambda),\ell^\infty(\Lambda)) \longrightarrow \mathcal{H}$ is bounded.
	
	\begin{theorem}\label{Th3.8}
		Let $\{g_j\}_{j \in \Lambda} \subset \mathcal{H}$ be a Bessel sequence with optimal Bessel bound $\beta>0$.  Then, for each $T=[a_{ij}] \in \mathcal{B}^{s}(\ell^2(\Lambda),\ell^\infty(\Lambda))$, the limit
		\begin{align*}
				\lim{T\{g_{j}\}_{j \in \Lambda}}=\lim_{\abs{i} \rightarrow \infty}{[a_{ij}]_{i,\,j \in \Lambda}\{g_{j}\}_{j \in \Lambda}}:=\lim_{\abs{i}\rightarrow\infty} \sum_{j \in \Lambda}a_{ij}g_{j}
		\end{align*}
	exists in $\mathcal{H}$. Moreover, the mapping
		\begin{align*}
			\mathscr{R}:\mathcal{B}^{s}(\ell^2(\Lambda),\ell^\infty(\Lambda)) \longrightarrow \mathcal{H}\,\, \text{defined as}\,\,
			T\mapsto \lim{T\{g_{j}\}_{j \in \Lambda}}
		\end{align*}
		is a well defined bounded operator whose norm is precisely $\sqrt{\beta}$.
	\end{theorem}
	
	\begin{proof}
		Let  $ T \in \mathcal{B}^{s}(\ell^2(\Lambda),\ell^\infty(\Lambda))$ be arbitrary and let  $ r_{i}=(\ldots,a_{i-2},a_{i-2+\frac{r}{N}},a_{i0},a_{i\frac{r}{N}},\ldots) $ represents the $i$-th row of $T$.	Write
		$(T\{g_{j}\}_{j \in \Lambda})_{i}:=\sum_{j \in \Lambda}a_{ij}g_{j} ,\, i \in \Lambda$.
		Then, we have
	\begin{align*}	
		\norm{(T\{g_{j}\}_{j \in \Lambda})_{i}}_{\mathcal{H}}  =  \Big\|\sum_{j \in \Lambda}a_{ij}g_{j}\Big\|_{\mathcal{H}}
		\leq  \sqrt{\beta}\norm{r_{i}}_{\ell^2(\Lambda)}
		\leq \sqrt{\beta}\norm{T}_{\ell^2(\Lambda) \rightarrow \ell^{\infty} (\Lambda)},	\, i \in \Lambda.
		\end{align*}
		Letting $\abs{i}\rightarrow \infty$, we have
		\begin{align*}
			\norm{\mathscr{R}(T)}_{\mathcal{H}}=\norm{ \lim{T\{g_{j}\}_{j \in \Lambda}}}_{\mathcal{H}}\leq \sqrt{\beta}\norm{T}_{\ell^2(\Lambda) \rightarrow \ell^{\infty} (\Lambda)}.	
		\end{align*}
     This implies that the operator $\mathscr{R}$ is bounded and
     $	\norm{\mathscr{R}}_{\mathcal{B}^{s}(\ell^2(\Lambda),\ell^\infty(\Lambda))\rightarrow \mathcal{H} }\leq \sqrt{\beta} $.	

	Now, for the reverse inequality, let $T \in \mathcal{B}^{s}(\ell^2(\Lambda),\ell^\infty(\Lambda))$ be such that all its rows are equal to a fixed vector
	\begin{align*}
		z=(\ldots,z_{-4},z_{-4+\frac{r}{N}},z_{-2},z_{-2+\frac{r}{N}},z_{0},z_{\frac{r}{N}},z_{2},z_{2+\frac{r}{N}},\ldots) \in \ell^2(\Lambda).
	\end{align*}
	Then,
	$\norm{T}_{\ell^2(\Lambda) \rightarrow \ell^{\infty} (\Lambda)}=\norm{z}_{\ell^2(\Lambda)} \quad
	\text{and} \quad
		\mathscr{R}(T)=\lim{T\{g_{j}\}_{j \in \Lambda}}= \sum_{j \in \Lambda}a_{j}g_{j}$.
		Hence, the inequality
		\begin{align*}
		\norm{\mathscr{R}(T)}_{\mathcal{H}} \leq \norm{\mathscr{R}}_{\mathcal{B}^{s}(\ell^2(\Lambda),\ell^\infty(\Lambda))\rightarrow \mathcal{H}} \norm{T}_{\ell^2(\Lambda) \rightarrow \ell^{\infty} (\Lambda)}
		\end{align*}
        converts into
        \begin{align*}
        \Big\|\sum_{j \in \Lambda}a_{j}g_{j}\Big\|_{\mathcal{H}}  \leq \norm{\mathscr{R}}_{\mathcal{B}^{s}(\ell^2(\Lambda),\ell^\infty(\Lambda))\rightarrow \mathcal{H}} \norm{z}_{\ell^2(\Lambda)}, \,\, z \in \ell^2(\Lambda).
        \end{align*}
		 But, since $z \in \ell^2(\Lambda)$ is arbitrary and $\{g_{j}\}_{j \in \Lambda}$ is a Bessel sequence with optimal Bessel bound $\beta>0$.
		This implies,
		$ \sqrt{\beta} \leq \norm{\mathscr{R}}_{\mathcal{B}^{s}(\ell^2(\Lambda),\ell^\infty(\Lambda))\rightarrow \mathcal{H}} $.
	This completes the proof.
		
		\end{proof}
		
	Next, we illustrate Theorem \ref{Th3.8} with the following example.
    \begin{example}
		Let $\mathcal{H}=\ell^2(\Lambda)$ and let $\{g_j\}_{j \in \Lambda}=\{e_{j}\}_{j \in \Lambda}$ be the standard orthonormal basis for $\ell^2(\Lambda)$ which forms a Bessel sequence with optimal Bessel bound $1$.
		Then, for each $T=[a_{ij}] \in \mathcal{B}^{s}(\ell^2(\Lambda),\ell^\infty(\Lambda))$, the limit
		\begin{align*}
				\lim{T\{g_{j}\}_{j \in \Lambda}}=\lim_{\abs{i}\rightarrow\infty} \sum_{j \in \Lambda}a_{ij}e_{j}
			= \lim_{\abs{i}\rightarrow \infty}r_{i},
		\end{align*}
		which exist by the definition of $\mathcal{B}^{s}(\ell^2(\Lambda),\ell^\infty(\Lambda))$. Thus, the mapping
		\begin{align*}
			\mathscr{R}:\mathcal{B}^{s}(\ell^2(\Lambda),\ell^\infty(\Lambda)) \longrightarrow \mathcal{H}\,\, \text{defined as} \,\,
			T\mapsto \lim{T\{g_{j}\}_{j \in \Lambda}}
		\end{align*}
		is well defined. Now, we compute
		\begin{align*}
			\norm{\mathscr{R}(T)}_{\mathcal{H}}=&\norm{ 	\lim{T\{g_{j}\}_{j \in \Lambda}}}_{\mathcal{H}}\\
			=& \Big\|\lim_{\abs{i}\rightarrow \infty}r_{i}\Big\|_{\ell^2(\Lambda)}\\
			=& \lim_{\abs{i}\rightarrow \infty}\norm{r_{i}}_{\ell^2(\Lambda)}\\
			\leq & \sup_{i \in \Lambda}\norm{r_{i}}_{\ell^2(\Lambda)}\\
			=&\norm{T}_{\ell^2(\Lambda) \rightarrow \ell^{\infty}(\Lambda)}.
		\end{align*}
		This implies, $\norm{\mathscr{R}} \leq 1$.
		Now, for the reverse inequality, let $T \in \mathcal{B}^{s}(\ell^2(\Lambda),\ell^\infty(\Lambda))$ be such that all its rows are equal to a fixed vector
		\begin{align*}
			z =(\ldots,0,0,0,0,1,0,0,0,\ldots) \in \ell^2(\Lambda).
		\end{align*}
         Then,
		$ \norm{\mathscr{R}(T)}_{\mathcal{H}}=\sup\limits_{i \in \Lambda}\norm{r_{i}}_{\ell^2(\Lambda)}
			= 1
			= \norm{T}_{\ell^2(\Lambda) \rightarrow \ell^{\infty} (\Lambda)} $.
			This implies that $\norm{\mathscr{R}} = 1$.
		\end{example}
	
	\subsection{Finite Time Iterations}\label{Sec.IV}
	In this subsection, we deal with the case where the reconstruction of the source term is achieved using a finite number of time iterations of the non-uniform discrete dynamical system \eqref{Dynamical System}. In this case,  all data measurements are performed by using the space $\mathcal{B}(\ell^2(\Lambda),\C^{[2K]})$.
	In our initial results, the source $w$ can be any element of the space $\mathcal{H}$. While this scenario is less practical compared to the more restricted case of closed subspaces, since, in reality, sources are typically confined to specific spatial regions, it offers a mathematically elegant solution to the source recovery problem.
	
	We begin this section by proving the following lemma, which verifies that the data matrix for finite time iterations belongs to the space $\mathcal{B}(\ell^2(\Lambda),\C^{[2K]})$.
	
	\begin{lemma}\label{L3.10}
		For the non-uniform discrete dynamical system \eqref{Dynamical System} with any initial states $x_{0},x_{-2} \in  \mathcal{H}$, and $1 \leq |[2K]| < \infty $, the data matrix $D(x_{0},x_{-2},w)=[\langle x_{\lambda},g_{j}\rangle]_{\lambda \in [2K],\,j \in \Lambda}  \in \mathcal{B}(\ell^2(\Lambda),\C^{[2K]})$.
	\end{lemma}
	
	\begin{proof}
		Since $\{g_j\}_{j \in \Lambda} \subset \mathcal{H} $ is a Bessel sequence  with Bessel bound $\beta>0$. Then,
		\begin{eqnarray*}
			\norm{D(x_{0},x_{-2},w)}_{\ell^2(\Lambda) \rightarrow \C^{[2K]}} = \sum_{\lambda \in [2K]}\sum_{j \in \Lambda}{\abs{\langle x_{\lambda},g_{j}\rangle}}^{2}
			 \leq  \beta\sum_{\lambda \in [2K]}{\norm{x_\lambda}}^{2}_{\mathcal{H}} < \infty.
		\end{eqnarray*}
		Hence, the data matrix $D(x_{0},x_{-2},w)=[\langle x_{\lambda},g_{j}\rangle]_{\lambda \in [2K],\,j \in \Lambda} \in \mathcal{B}(\ell^2(\Lambda),\C^{[2K]})$.
	\end{proof}
	
	The next proposition provides  conditions for the non-uniform discrete dynamical system \eqref{Dynamical System}, under which the image of data operator becomes a closed subspace of $\mathcal{B}(\ell^2(\Lambda),\C^{[2K]})$.
	
	\begin{proposition}\label{P3.11}
		For the non-uniform discrete dynamical system \eqref{Dynamical System} with any initial states $x_{0},x_{-2} \in  \mathcal{H}$, and $1 \leq |[2K]| < \infty$, the image of the data operator $D$ defined by
		$$ D:\mathcal{H} \times \mathcal{H} \times \mathcal{H} \longrightarrow \mathscr{B}(\ell^2(\Lambda),\C^{[2K]}) $$
		$$D(x_{0},x_{-2},w):=[\langle x_{\lambda},g_{j}\rangle]_{\lambda \in [2K],\,j \in \Lambda} $$
		is a closed subspace in $\mathcal{B}(\ell^2(\Lambda),\C^{[2K]})$ if $\{g_j\}_{j \in \Lambda}$ is a frame for $\mathcal{H}$.
	\end{proposition}
	
	\begin{proof}
		Let $\{(x_{0}^{k},x_{-2}^{k},w^{k})\}_{k \in \N }$ be a sequence in $\mathcal{H} \times\mathcal{H} \times \mathcal{H}$ such that
		 $\{D(x_{0}^{k},x_{-2}^{k},w^{k})\}_{k \in \N}$ converges to $T \in \mathcal{B}(\ell^2(\Lambda),\C^{|[2N]|})$, i.e.,
		 		$ \lim_{k \rightarrow \infty}D(x_{0}^{k},x_{-2}^{k},w^{k})=T $.
	  By the definition of $\mathcal{B}(\ell^2(\Lambda),\C^{[2K]})$, $T$ is of the form
	  \begin{align*}
	  	T =(r_{-2K},r_{-2K+\frac{r}{N}},\cdots, r_{-2},r_{-2+\frac{r}{N}},r_{0},r_{\frac{r}{N}},\cdots,r_{2K-2},r_{2K-2+\frac{r}{N}}),
	  \end{align*}
		where, $r_{i}$ denotes the rows of $T$, and for each $\lambda \in [2K]$,
		\begin{align*}
			r_{\lambda}:=(a_{\lambda-2K},a_{\lambda-2K+\frac{r}{N}},\cdots, a_{\lambda-2},a_{\lambda-2+\frac{r}{N}},a_{\lambda0},a_{\lambda\frac{r}{N}},\cdots,a_{\lambda2K-2},a_{\lambda2K-2+\frac{r}{N}}) \in \ell^2(\Lambda).
		\end{align*}
       Since $\{g_j\}_{j \in \Lambda}$  is a frame for $\mathcal{H}$. So, the image of the analysis operator $$R:\mathcal{H} \longrightarrow \ell^2(\Lambda)$$
		$$R(h) ={(\langle h,g_{j}\rangle)}_{j \in \Lambda}$$
		is a closed linear subspace of $\ell^2(\Lambda)$. Note that,
		\begin{align*}
		r_{0}=\lim_{k \rightarrow \infty}R(x_{0}^{k}) \quad
		\text{and} \quad
		r_{-2}=\lim_{k \rightarrow \infty}R(x_{-2}^{k}),
		\end{align*}
		where the above limits exist in $\ell^2(\Lambda)$-norm. Since the image of $R$ is closed, there exist
		$x_{0}$ and $x_{-2}$ in $\mathcal{H}$ such that
		\begin{align}
			r_{0}=R(x_{0})= {(\langle x_{0},g_{j}\rangle)}_{j \in \Lambda}
			\quad 	\text{and} \quad
			r_{-2}=R(x_{-2}):={(\langle x_{-2},g_{j}\rangle)}_{j \in \Lambda}.
			\label{eq:3.5}
		\end{align}
		Since $\{g_j\}_{j \in \Lambda}$ is a frame with optimal lower and upper bounds $\alpha$  and $\beta$, respectively.
Thus,
\begin{align*}
			{\norm{x_{0}^{k}-x_{0}}}^{2}_{\mathcal{H}} \leq  \frac{1}{\alpha}\sum_{j \in \Lambda}{\abs{\langle x_{0}^{k}-x_{0},g_{j}\rangle}}^{2}
			=\frac{1}{\alpha}	{\Big\|R(x_{0}^{k})-\underbrace{R(x_{0})}_{r_{0}}}\Big\|^{2}_{\ell^2(\Lambda)},
\intertext{and}
			{\norm{x_{-2}^{k}-x_{-2}}}^{2}_{\mathcal{H}} \leq  \frac{1}{\alpha}\sum_{j \in \Lambda}{\abs{\langle x_{-2}^{k}-x_{-2},g_{j}\rangle}}^{2}
			=\frac{1}{\alpha}	{\Big\|R(x_{-2}^{k})-\underbrace{R(x_{-2})}_{r_{-2}}}\Big\|^{2}_{\ell^2(\Lambda)},
\end{align*}
		and thus by using \eqref{eq:3.5}, we have
			\begin{align}
			\notag
			\lim_{k \rightarrow \infty}{\norm{x_{0}^{k}-x_{0}}}^{2}_{\mathcal{H}}  & \leq \lim_{k \rightarrow \infty}\frac{1}{\alpha}	{\norm{R(x_{0}^{k})-{R(x_{0})}}}^{2}_{\ell^2(\Lambda)}=0, \\
			\notag
 \intertext{and}
			\lim_{k \rightarrow \infty}{\norm{x_{-2}^{k}-x_{-2}}}^{2}_{\mathcal{H}} & \leq \lim_{k \rightarrow \infty}\frac{1}{\alpha}	{\norm{R(x_{-2}^{k})-{R(x_{-2  })}}}^{2}_{\ell^2(\Lambda)}=0.
			\label{eq:3.6}
		\end{align}
This implies that $\{x_{0}^{k}\}_{k \in \N}$ converges to $x_{0}$ and $\{x_{-2}^{k}\}_{k \in \N}$ converges to $x_{-2}$.
		Similarly,
		\begin{align*}
			r_{\frac{r}{N}}=\lim_{k \rightarrow \infty}R(x_{\frac{r}{N}}^{k})
			= \lim_{k \rightarrow \infty}R(Ax_{0}^{k}+ w^{k})
			=R(Ax_{0})+\lim_{k \rightarrow \infty}R(w^{k}),
	\intertext{and}
		r_{-2+\frac{r}{N}}=\lim_{k \rightarrow \infty}R(x_{-2+\frac{r}{N}}^{k})
			= \lim_{k \rightarrow \infty}R(Ax_{-2}^{k}+ w^{k})
			=R(Ax_{-2})+\lim_{k \rightarrow \infty}R(w^{k}).
		\end{align*}
		In particular, $\{R(w^{k})\}_{k \in \N}$ converges in $\ell^2(\Lambda)$, and since the image of $R$ is closed, there exists
		$w \in  \mathcal{H}$ such that
		\begin{align*}
		\lim_{k \rightarrow \infty}	{\norm{R(w^{k})-{R(w)}}}^{2}_{\ell^2(\Lambda)}=0.
		\end{align*}
	By using the same argument as in \eqref{eq:3.6}, we have that $\{w^{k}\}_{k \in \N}$ converges to $w$ in $\mathcal{H}$.
		Thus,
		\begin{align*}
			r_{\frac{r}{N}}=R(Ax_{0}+w) \quad
			\text{and} \quad
			r_{-2+\frac{r}{N}}=R(Ax_{-2}+w).
		\end{align*}
	Finally, as the $\lambda$-th state of the non-uniform discrete dynamical system \eqref{Dynamical System}, with arbitrary initial states $x_{0}^{k},x_{-2}^{k} \in  \mathcal{H}$ and source term
		$w^{k} \in W $ can be expressed as
		\[
		x_{\lambda}^{k}=
		\begin{cases}
			A^{\lambda}x_{0}^{k}+(I+A+A^{2}+\cdots+A^{\lambda-1})w^{k}, \,\, \lambda \in  2\Z^{+}; &\\
			A^{\lambda+1-\frac{r}{N}}x_{0}^{k}+(I+A+A^{2}+\cdots+A^{\lambda-\frac{r}{N}})w^{k},\,\, \lambda \in  2\Z^{+}+\frac{r}{N} \cup \left\{\frac{r}{N}\right\}; &\\
			A^{-\lambda-2}x_{-2}^{k}+(I+A+A^{2}+\cdots+A^{-\lambda-3})w^{k},\,\, \lambda \in  2\Z^{-}; &\\
			A^{-\lambda-1+\frac{r}{N}}x_{-2}^{k}+(I+A+A^{2}+\cdots+A^{-\lambda-2+\frac{r}{N}})w^{k},\,\, \lambda \in  2\Z^{-}+\frac{r}{N},
		\end{cases}
		\]
		and similarly, the $\lambda$-th of  the non-uniform discrete dynamical system \eqref{Dynamical System}, with arbitrary initial states $x_{0},x_{-2} \in  \mathcal{H}$ and source term
		$w \in W $ can be expressed as
		\[
		x_{\lambda}=
		\begin{cases}
			A^{\lambda}x_{0}+(I+A+A^{2}+\cdots+A^{\lambda-1})w ,\,\, \lambda \in  2\Z^{+}; &\\
			A^{\lambda+1-\frac{r}{N}}x_{0}+(I+A+A^{2}+\cdots+A^{\lambda-\frac{r}{N}})w ,\,\, \lambda \in  2\Z^{+}+\frac{r}{N} \cup \left\{\frac{r}{N}\right\};&\\
			A^{-\lambda-2}x_{-2}+(I+A+A^{2}+\cdots+A^{-\lambda-3})w,\,\, \lambda \in  2\Z^{-};&\\
			A^{-\lambda-1+\frac{r}{N}}x_{-2}+(I+A+A^{2}+\cdots+A^{-\lambda-2+\frac{r}{N}})w,\,\, \lambda \in  2\Z^{-}+\frac{r}{N}.
		\end{cases}
		\]
		Therefore, we have
		\begin{align*}
			r_{\lambda} =&\lim_{k \rightarrow \infty}R(x_{\lambda}^{k}) \; \text{for all} \; \lambda \in \Lambda\\
			=&	\begin{cases}
				\lim_{k \rightarrow \infty}R(A^{\lambda}x_{0}^{k}+(I+A+A^{2}+\cdots+A^{\lambda-1})w^{k}),\,\,\lambda \in  2\Z^{+};&\\
				\lim_{k \rightarrow \infty}R(A^{\lambda+1-\frac{r}{N}}x_{0}^{k}+(I+A+A^{2}+\cdots+A^{\lambda-\frac{r}{N}})w^{k}),\,\, \lambda \in  2\Z^{+}+\frac{r}{N} \cup \left\{\frac{r}{N}\right\};&\\
				\lim_{k \rightarrow \infty}R(A^{-\lambda-2}x_{-2}^{k}+(I+A+A^{2}+\cdots+A^{-\lambda-3})w^{k}),\,\, \lambda \in  2\Z^{-};&\\
				\lim_{k \rightarrow \infty}R(A^{-\lambda-1+\frac{r}{N}}x_{-2}^{k}+(I+A+A^{2}+\cdots+A^{-\lambda-2+\frac{r}{N}})w^{k}),\,\, \lambda \in  2\Z^{-}+\frac{r}{N},
			\end{cases}\\
			=&	\begin{cases}
				R(A^{\lambda}x_{0}+(I+A+A^{2}+\cdots+A^{\lambda-1})w),\,\, \lambda \in  2\Z^{+};&\\
				R(A^{\lambda+1-\frac{r}{N}}x_{0}+(I+A+A^{2}+\cdots+A^{\lambda-\frac{r}{N}})w),\,\, \lambda \in  2\Z^{+}+\frac{r}{N} \cup \left\{\frac{r}{N}\right\};&\\
				R(A^{-\lambda-2}x_{-2}+(I+A+A^{2}+\cdots+A^{-\lambda-3})w),\,\,\lambda \in  2\Z^{-};&\\
				R(A^{-\lambda-1+\frac{r}{N}}x_{-2}+(I+A+A^{2}+\cdots+A^{-\lambda-2+\frac{r}{N}})w),\,\,\lambda \in  2\Z^{-}+\frac{r}{N},
			\end{cases}\\
			 = & R(x_{\lambda}), \, \lambda \in \Lambda.
		\end{align*}
		This implies that
	$	\lim_{k \rightarrow \infty}D(x_{0}^{k},x_{-2}^{k},w^{k})=T=	D(x_{0},x_{-2},w) $.
	Thus, $T$ belongs to the image of the data operator $D$. Hence, the image of $D$ is a closed subspace in $\mathcal{B}(\ell^2(\Lambda),\C^{[2K]})$.
	\end{proof}
	
	Taking exhortation from  \cite[Theorem 3.2]{CFSRDS}, the following theorem incorporates a necessary and sufficient condition on $\{g_j\}_{j \in \Lambda} \subset \mathcal{H} $ for the existence of the reconstruction operator $\mathscr{R}:\mathcal{B}(\ell^2(\Lambda),\C^{[2K]}) \longrightarrow \mathcal{H}$.
	
	\begin{theorem}\label{Th3.12}
		Let $\{g_j\}_{j \in \Lambda} \subset \mathcal{H} $ be a Bessel sequence with Bessel bound  $\beta>0$. Then, for the non-uniform discrete dynamical system \eqref{Dynamical System}, with any arbitrary initial states $x_{0},x_{-2} \in  \mathcal{H}$, the source term
		$w \in W$ can be stably recovered from the measurements $D(x_{0},x_{-2},w)=[\langle x_{\lambda},g_{j}\rangle]_{\lambda \in [2K],\,j \in \Lambda}$ for some $1 \leq |[2K]| < \infty$ if and only if $ \{g_j\}_{j \in \Lambda}$ is a frame for  $\mathcal{H}$.
	\end{theorem}
	
	\begin{proof}
		First, assume that the source term
		$w \in W$ can be stably  recovered from the measurements $D(x_{0},x_{-2},w)=[\langle x_{\lambda},g_{j}\rangle]_{\lambda \in [2K],\,j \in \Lambda}$ in $|[2K]|$ time of iterations. That is, there exists a bounded linear operator
		$\mathscr{R}:\mathcal{B}(\ell^2(\Lambda),\C^{[2K]}) \longrightarrow \mathcal{H}$ such that
		\begin{align*}
			\mathscr{R}(D(x_{0},x_{-2},w))=w \; \text{for all} \; x_{0},x_{-2},w  \in \mathcal{H}.
		\end{align*}
Hence, there exists a positive constant $c$ such
		\begin{align}
			& \norm{\mathscr{R}(T)} \leq c\sum_{\lambda \in [2K]}\sum_{j \in \Lambda}\abs{a_{\lambda j}}^{2}, \,\, T \in \mathcal{B}^{s}(\ell^2(\Lambda),\C^{[2K]}),
			\label{eq:3.7}
		\intertext{and}
		& \mathscr{R}(D(x_{0},x_{-2},w))=w, \, x_{0},x_{-2}, w \in \mathcal{H}.
			\label{eq:3.8}
		\end{align}
		Our aim is to prove that $\{g_j\}_{j \in \Lambda}$ is a frame for  $\mathcal{H}$.
		By using \eqref{eq:3.7} and \eqref{eq:3.8}, we compute
		\begin{align}
			\notag
			{\norm{w}}^{2}_{\mathcal{H}} \leq & c\sum_{\lambda \in [2K]}\sum_{j \in \Lambda}\abs{\langle x_{\lambda},g_{j}\rangle}^{2}\\
			\notag
			= & c\sum_{j \in \Lambda}\sum_{\lambda \in [2K]}\abs{\langle x_{\lambda},g_{j}\rangle}^{2}\\
		\notag
			=& c \Big[\sum_{j \in \Lambda}\Big\{\sum_{\lambda \in [2K] \cap 2\Z^{+}}\abs{\langle x_{\lambda},g_{j}\rangle}^{2}+\sum_{\lambda \in [2K] \cap 2\Z^{+}+\frac{r}{N} \cup \left\{\frac{r}{N}\right\}}\abs{\langle x_{\lambda},g_{j}\rangle}^{2} \Big. \Big. \\
			\notag
			&
			\Big. \Big. + \sum_{\lambda \in [2K] \cap 2\Z^{-}}\abs{\langle x_{\lambda},g_{j}\rangle}^{2}+\sum_{\lambda \in [2K] \cap 2\Z^{-}+\frac{r}{N}}\abs{\langle x_{\lambda},g_{j}\rangle}^{2}\Big\}\Big]\\
			\notag
			=& c \Big[\sum_{j \in \Lambda}\Big\{\sum_{\lambda \in [2K] \cap 2\Z^{+}}\Big|\langle A^{\lambda}x_{0}
          +\sum_{k=0}^{\lambda-1}A^{k}w,g_{j}\rangle \Big|^{2}  \Big. \Big. \\
          \notag
          & +\sum_{\lambda \in [2K] \cap 2\Z^{+}+\frac{r}{N} \cup \left\{\frac{r}{N}\right\}}\Big|\langle A^{\lambda+1-\frac{r}{N}}x_{0}+\sum_{k=0}^{\lambda-\frac{r}{N}}A^{k}w,g_{j}\rangle\Big|^{2} \\
         \notag
			& + \sum_{\lambda \in [2K] \cap 2\Z^{-}}\Big|\langle A^{-\lambda-2}x_{-2}+\sum_{k=0}^{-\lambda-3}A^{k}w,g_{j}\rangle\Big|^{2} \\
			 & \Big. \Big.+\sum_{\lambda \in [2K] \cap 2\Z^{-}+\frac{r}{N}}\Big|\langle A^{-\lambda-1+\frac{r}{N}}x_{-2}+\sum_{k=0}^{-\lambda-2+\frac{r}{N}}A^{k}w,g_{j}\rangle\Big|^{2}\Big\}\Big].
			\label{eq:3.9}
		\end{align}
		Since, the equation \eqref{eq:3.9} holds for every $x_{0},x_{-2} \in \mathcal{H}$. So, in particular for $x_{0}=x_{-2}$, choose $w=(I-A)x_{0}$.
		So, the above equation can be written as
		\begin{align*}
			{\norm{(I-A)x_{0}}}^{2}_{\mathcal{H}} \leq& c \Big[\sum_{j \in \Lambda}\Big\{\sum_{\lambda \in [2K] \cap 2\Z^{+}}\abs{\langle x_{0},g_{j}\rangle}^{2}+\sum_{\lambda \in [2K] \cap 2\Z^{+}+\frac{r}{N} \cup \left\{\frac{r}{N}\right\}}\abs{\langle x_{0},g_{j}\rangle}^{2} \Big. \Big. \\
			&
		\Big. \Big. + \sum_{\lambda \in [2K] \cap 2\Z^{-}}\abs{\langle x_{0},g_{j}\rangle}^{2}+\sum_{\lambda \in [2K] \cap 2\Z^{-}+\frac{r}{N}}\abs{\langle x_{0},g_{j}\rangle}^{2}\Big\}\Big]\\
		= & c \sum_{j \in \Lambda}\sum_{\lambda \in [2K]}\abs{\langle x_{0},g_{j}\rangle}^{2}\\
		= & |[2K]| c \sum_{j \in \Lambda}\abs{\langle x_{0},g_{j}\rangle}^{2}.
		\end{align*}
	Repeated use of this estimation implies
		\begin{align}
			\notag
			{\norm{(I-A)^{k}x_{0}}}^{2}_{\mathcal{H}} \leq & |[2K]| c \sum_{j \in \Lambda}\abs{\langle (I-A)^{k-1}x_{0},g_{j}\rangle}^{2}\\
			\notag	
			\leq
	&    |[2K]| c \beta \norm{(I-A)^{k-1}x_{0}}^{2}\\
	\notag
	\leq & {|[2K]|}^{2}c^2 \beta \sum_{j \in \Lambda}\abs{\langle (I-A)^{k-2}x_{0},g_{j}\rangle}^{2}\\
	\notag
	& \vdots \\	
	\leq & c_{k} \sum_{j \in \Lambda}\abs{\langle x_{0},g_{j}\rangle}^{2}, \; x_{0} \in \mathcal{H},
	\end{align}
where $c_{k}={|[2K]|}^{k}c^{k}{\beta}^{k-1}$.  Write $A=I-(I-A)$. Then, we observe that
    \begin{align*}
	 	\sum_{k=0}^{\lambda -1}A^{k}w= & \sum_{k=0}^{\lambda -1}\alpha_{n,k}(I-A)^{k}w, \,\, \lambda \in  2\Z^{+};\\
	 	\sum_{k=0}^{\lambda -\frac{r}{N}}A^{k}w= & \sum_{k=0}^{\lambda -\frac{r}{N}}\beta_{n,k}(I-A)^{k}w,\,\, \lambda \in  2\Z^{+}+\frac{r}{N} \cup \left\{\frac{r}{N}\right\}; \\
	 	\sum_{k=0}^{-\lambda-3}A^{k}w =& \sum_{k=0}^{-\lambda-3}\gamma_{n,k}(I-A)^{k}w,\,\, \lambda \in  2\Z^{-};\\
	 		\sum_{k=0}^{-\lambda-2+\frac{r}{N}}A^{k}w = & \sum_{k=0}^{-\lambda-2+\frac{r}{N}}\delta_{n,k}(I-A)^{k}w,\,\, \lambda \in  2\Z^{-}+\frac{r}{N},
	 \end{align*}
	 where $\alpha_{n,k},\beta_{n,k},\gamma_{n,k}$ and $\delta_{n,k}$ are binomial coefficients.
	  We go back to \eqref{eq:3.9} again, and put $x_{0}=x_{-2}=0$. Then, we compute
	\begin{align*}
	 	{\norm{w}}^{2}_{\mathcal{H}} \leq &  c \Big[\sum_{j \in \Lambda}\Big\{\sum_{\lambda \in [2K] \cap 2\Z^{+}}\Big|\langle \sum_{k=0}^{\lambda-1}A^{k}w,g_{j}\rangle\Big|^{2}+\sum_{\lambda \in [2K] \cap 2\Z^{+}+\frac{r}{N} \cup \left\{\frac{r}{N}\right\}}\Big|\langle \sum_{k=0}^{\lambda-\frac{r}{N}}A^{k}w,g_{j}\rangle\Big|^{2} \Big. \Big. \\
	 	& +\Big. \Big. \sum_{\lambda \in [2K] \cap 2\Z^{-}}\Big|\langle \sum_{k=0}^{-\lambda-3}A^{k}w,g_{j}\rangle\Big|^{2}+\sum_{\lambda \in [2K] \cap 2\Z^{-}+\frac{r}{N}}\Big|\langle \sum_{k=0}^{-\lambda-2+\frac{r}{N}}A^{k}w,g_{j}\rangle\Big|^{2}\Big\}\Big]\\
	 	= &  c \Big[\sum_{j \in \Lambda}\Big\{\sum_{\lambda \in [2K] \cap 2\Z^{+}}\Big|\langle \sum_{k=0}^{\lambda -1}\alpha_{n,k}(I-A)^{k}w,g_{j}\rangle\Big|^{2} \Big. \Big. \\
	 	& +\sum_{\lambda \in [2K] \cap 2\Z^{+}+\frac{r}{N} \cup \left\{\frac{r}{N}\right\}}\Big|\langle \sum_{k=0}^{\lambda -\frac{r}{N}}\beta_{n,k}(I-A)^{k}w,g_{j}\rangle\Big|^{2} +\sum_{\lambda \in [2K] \cap 2\Z^{-}}\Big|\langle \sum_{k=0}^{-\lambda-3}\gamma_{n,k}(I-A)^{k}w,g_{j}\rangle\Big|^{2} \\
	  	& \Big. \Big. +\sum_{\lambda \in [2K] \cap 2\Z^{-}+\frac{r}{N}}\Big|\langle \sum_{k=0}^{-\lambda-2+\frac{r}{N}}\delta_{n,k}(I-A)^{k}w,g_{j}\rangle\Big|^{2}\Big\}\Big]\\
	 	= &  c \Big[\sum_{j \in \Lambda}\Big\{{(\sup\abs{\alpha_{n,k}})}^{2}\sum_{\lambda \in [2K] \cap 2\Z^{+}}\Big|\langle \sum_{k=0}^{\lambda -1}(I-A)^{k}w,g_{j}\rangle\Big|^{2} \Big. \Big.\\
	 	& + {(\sup\abs{\beta_{n,k}})}^{2}\sum_{\lambda \in [2K] \cap 2\Z^{+}+\frac{r}{N} \cup \left\{\frac{r}{N}\right\}}\Big|\langle \sum_{k=0}^{\lambda -\frac{r}{N}}(I-A)^{k}w,g_{j}\rangle\Big|^{2} \\
	 	&  + {(\sup\abs{\gamma_{n,k}})}^{2} \sum_{\lambda \in [2K] \cap 2\Z^{-}}\Big|\langle \sum_{k=0}^{-\lambda-3}(I-A)^{k}w,g_{j}\rangle\Big|^{2}+ \\
	 	& \Big. \Big. +   {(\sup\abs{\delta_{n,k}})}^{2}\sum_{\lambda \in [2K] \cap 2\Z^{-}+\frac{r}{N}}\Big|\langle \sum_{k=0}^{-\lambda-2+\frac{r}{N}}(I-A)^{k}w,g_{j}\rangle\Big|^{2}\Big\}\Big]\\
	 	\leq & \, c \, \max\{a',b',c',d'\}\Big[\sum_{j \in \Lambda}\Big\{\sum_{\lambda \in [2K] \cap 2\Z^{+}}\Big|\langle \sum_{k=0}^{\lambda -1}(I-A)^{k}w,g_{j}\rangle\Big|^{2} \Big.\Big.\\
	 	& +\sum_{\lambda \in [2K] \cap 2\Z^{+}+\frac{r}{N} \cup \left\{\frac{r}{N}\right\}}\Big|\langle \sum_{k=0}^{\lambda -\frac{r}{N}}(I-A)^{k}w,g_{j}\rangle\big|^{2}+ \sum_{\lambda \in [2K] \cap 2\Z^{-}}\Big|\langle \sum_{k=0}^{-\lambda-3}(I-A)^{k}w,g_{j}\rangle\Big|^{2}  \\
	 	& \Big. \Big. + \sum_{\lambda \in [2K] \cap 2\Z^{-}+\frac{r}{N}}\Big|\langle \sum_{k=0}^{-\lambda-2+\frac{r}{N}}(I-A)^{k}w,g_{j}\rangle\Big|^{2}\Big\}\Big],\\
& \Big(\text{where} \, \,		a'={(\sup\abs{\alpha_{n,k}})}^{2},
	 		b'={(\sup\abs{\beta_{n,k}})}^{2},
	 		c'={(\sup\abs{\gamma_{n,k}})}^{2}\, \text{and} \,
	 		d'={(\sup\abs{\delta_{n,k}})}^{2}\Big)\\
			= &  c \max\{a',b',c',d'\}\Big[\sum_{j \in \Lambda}\Big\{ \sum_{k=0}^{1}\abs{\langle(I-A)^{k}w,g_{j}\rangle}^{2}+\sum_{k=0}^{3}\abs{\langle(I-A)^{k}w,g_{j}\rangle}^{2}+\cdots \Big. \Big. \\
			&  + \sum_{k=0}^{2K-3}\abs{\langle(I-A)^{k}w,g_{j}\rangle}^{2}+ \sum_{k=0}^{0}\abs{\langle(I-A)^{k}w,g_{j}\rangle}^{2} +\sum_{k=0}^{2}\abs{\langle(I-A)^{k}w,g_{j}\rangle}^{2} + \cdots  \\
			& + \sum_{k=0}^{2K-2}\abs{\langle(I-A)^{k}w,g_{j}\rangle}^{2}+ \cdots \sum_{k=0}^{1}\abs{\langle(I-A)^{k}w,g_{j}\rangle}^{2}+\sum_{k=0}^{3}\abs{\langle(I-A)^{k}w,g_{j}\rangle}^{2}+\cdots  \\
			& \Big. \Big.  +\sum_{k=0}^{2K-3}\abs{\langle(I-A)^{k}w,g_{j}\rangle}^{2}+\sum_{k=0}^{0}\abs{\langle(I-A)^{k}w,g_{j}\rangle}^{2} +\sum_{k=0}^{2}\abs{\langle(I-A)^{k}w,g_{j}\rangle}^{2} + \cdots\\
			& + \sum_{k=0}^{2K-2}\abs{\langle(I-A)^{k}w,g_{j}\rangle}^{2}
			\Big\}\Big]\\
			\leq & 2c \max\{a',b',c',d'\} (2K-1)\sum_{j \in \lambda}\sum_{k=0}^{2K-2}\abs{\langle(I-A)^{k}w,g_{j}\rangle}^{2}\\
			 \leq & \xi \beta\sum_{k=0}^{2K-2}{\norm{(I-A)^{k}w}}^{2}_{\mathcal{H}} + \xi \sum_{j \in \lambda}\abs{\langle w,g_{j}\rangle}^{2}\\
			 \leq & \xi \beta\sum_{k=0}^{2K-2}c_{k} \sum_{j \in \Lambda}\abs{\langle w,g_{j}\rangle}^{2} + \xi \sum_{j \in \lambda}\abs{\langle w,g_{j}\rangle}^{2}\\
			 = & \xi \Big(\beta\sum_{k=0}^{2K-2}c_{k}+1\Big) \sum_{j \in \lambda}\abs{\langle w,g_{j}\rangle}^{2}\\
			 \leq & \xi'\sum_{j \in \lambda}\abs{\langle w,g_{j}\rangle}^{2},
		\end{align*}
		where $\xi =2c\max\{a',b',c',d'\} (2K-1) $ and $\xi'=\xi \Big(\beta\sum_{k=0}^{2K-2}c_{k}+1\Big)$.
		Therefore,
		\begin{align*}
			\frac{1}{\xi'}{\norm{w}}^{2}_{\mathcal{H}} \leq  \sum_{j \in \Lambda}\abs{\langle w,g_{j}\rangle}^{2} \leq \beta{\norm{w}}^{2}_{\mathcal{H}} \; \text{for all} \; w \in \mathcal{H}.
		\end{align*}
		Hence, $\{g_j\}_{j \in \Lambda}$ is a frame for  $\mathcal{H}$.
		
		 Conversely, let us assume that  $\{g_j\}_{j \in \Lambda}$ is a frame for  $\mathcal{H}$ with frame bounds $\alpha,\beta> 0 $.	Let $\{\tilde{g_i}\}_{i \in \Lambda}$ be the canonical dual frame of $\{g_j\}_{j \in \Lambda}$. Then, for each $j \in \Lambda$,
		  	\begin{align}
		  		A^{*}g_{j}=\sum_{i \in \Lambda}c_{ij}g_{i} \; \text{where}, \text{each} \; c_{ij}=\langle A^{*}g_{j},\tilde{g_i}\rangle,\, i \in \Lambda.
		  		\label{eq:4.1}
		  	\end{align}
		From the non-uniform discrete dynamical system \eqref{Dynamical System}, we have
		  	\begin{align*}
		  	 w= &	x_{\lambda+\frac{r}{N}} -Ax_{\lambda},\,\, \lambda \in 2\Z; \\
		  	w= & x_{\lambda+2-\frac{r}{N}} - Ax_{\lambda},\,\, \lambda \in 2\Z^{+}+\frac{r}{N} \cup \left\{\frac{r}{N}\right\}; \\
		  	w=&	x_{\lambda-2-\frac{r}{N}} - Ax_{\lambda}, \,\, \lambda \in 2\Z^{-}+\frac{r}{N}.
		  	\end{align*}
		\textbf{Case-(i)} When $\lambda \in 2\Z$, we compute
		  	\begin{align}
		  		\notag
		  		\langle w, g_{j} \rangle =& 	\langle x_{\lambda+\frac{r}{N}} -Ax_{\lambda}, g_{j} \rangle\\
		  		\notag
		  		= & \langle x_{\lambda+\frac{r}{N}}, g_{j} \rangle - \langle Ax_{\lambda}, g_{j} \rangle\\
		  		\notag
		  		= & \langle x_{\lambda+\frac{r}{N}}, g_{j} \rangle - \langle x_{\lambda}, A^{*}g_{j} \rangle\\
		  		\notag
		  		= & \langle x_{\lambda+\frac{r}{N}}, g_{j} \rangle - \langle x_{\lambda}, \sum_{i \in \Lambda}c_{ij}g_{i} \rangle\\
		  		= & \langle x_{\lambda+\frac{r}{N}}, g_{j} \rangle - \sum_{i \in \Lambda}\bar{c_{ij}}\langle x_{\lambda}, g_{i} \rangle.
		  		\label{eq:3.12}
		  	\end{align}
		  Note that the equation \eqref{eq:3.12} satisfies for any two consecutive states $x_{\lambda}$  and $x_{\lambda+\frac{r}{N}}, \;\lambda \in 2\Z$.
		  Since $\{g_j\}_{j \in \Lambda}$ is a frame for  $\mathcal{H}$. Therefore, $w \in \mathcal{H}$ can be expressed as
	\begin{align*}
		w=\sum_{j \in \Lambda  }\langle w, g_{j} \rangle \tilde{g_{j}}.
	\end{align*}
By the equation \eqref{eq:3.12}, every coefficient $\langle w, g_{j} \rangle$ can be expressed in terms of the measurements $\{\langle x_{\lambda},g_{j}\rangle , \langle x_{\lambda+ \frac{r}{N}},g_{j}\rangle\}_{j \in \Lambda}$ for all values of $\lambda \in 2\Z$.
Therefore, we have
	\begin{align}
		w= \sum_{j \in \Lambda}\Big( \langle x_{\lambda+\frac{r}{N}}, g_{j} \rangle - \sum_{i \in \Lambda}\bar{c_{ij}}\langle x_{\lambda}, g_{i} \rangle\Big)\tilde{g_{j}}.
		\label{eq:3.13}
	\end{align}
In particular, if we take the space samples of the two states of the system $x_{0}$ and $x_{\frac{r}{N}}$, then we get
	\begin{align}
		w= \sum_{j \in \Lambda}\Big( \langle x_{\frac{r}{N}}, g_{j} \rangle - \sum_{i \in \Lambda}\bar{c_{ij}}\langle x_{0}, g_{i} \rangle\Big)\tilde{g_{j}}.
		\label{eq:3.14}
	\end{align}
	Now, from the equation \eqref{eq:3.12}, we have
	\begin{align}
		\notag
		\sum_{j \in \Lambda}{\Big|\sum_{i \in \Lambda}\bar{c_{ij}}\langle x_{0}, g_{i} \rangle\Big|}^{2} \leq & \sum_{j \in \Lambda}{\abs{\langle Ax_{0}, g_{j} \rangle}}^{2}\\
		\notag
		\leq & \beta{\norm{Ax_{0}}}^{2}_{\mathcal{H}}\\
		\notag
		\leq & \beta\norm{A}^{2}_{\mathcal{H} \rightarrow \mathcal{H}}{\norm{x_{0}}}^{2}_{\mathcal{H}}\\
		\leq & \frac{\beta}{\alpha}\norm{A}^{2}_{\mathcal{H} \rightarrow \mathcal{H}}\sum_{j \in \Lambda}{\abs{\langle x_{0}, g_{j} \rangle}}^{2},
		\label{eq:3.15}
		\end{align}
	and thus
	\begin{align}
		\notag
		{\Big\|\sum_{j \in \Lambda}\Big( \langle x_{\frac{r}{N}}, g_{j} \rangle - \sum_{i \in \Lambda}\bar{c_{ij}}\langle x_{0}, g_{i} \rangle\Big)\tilde{g_{j}}\Big\|}^{2}_{\mathcal{H}} \leq & \frac{1}{\alpha} \sum_{j \in \Lambda}{\Big|\langle x_{\frac{r}{N}}, g_{j} \rangle - \sum_{i \in \Lambda}\bar{c_{ij}}\langle x_{0}, g_{i} \rangle\Big|}^{2}\\
		\notag
		\leq & \frac{2}{\alpha} \sum_{j \in \Lambda}\Big({\abs{\langle x_{\frac{r}{N}}, g_{j} \rangle}}^{2} + {\Big|\sum_{i \in \Lambda}\bar{c_{ij}}\langle x_{0}, g_{i} \rangle\Big|}^{2}\Big)\\
		\notag
		\leq & \frac{2}{\alpha} \sum_{j \in \Lambda}{\abs{\langle x_{\frac{r}{N}}, g_{j} \rangle}}^{2} + 2\frac{\beta}{\alpha^{2}}\norm{A}^{2}_{\mathscr{H} \rightarrow \mathcal{H}}\sum_{j \in \Lambda}{\abs{\langle x_{0}, g_{j} \rangle}}^{2}\\
		\leq & \tilde{c}\Big(\sum_{j \in \Lambda}{\abs{\langle x_{\frac{r}{N}}, g_{j} \rangle}}^{2} +\sum_{j \in \Lambda}{\abs{\langle x_{0}, g_{j} \rangle}}^{2} \Big),
		\label{eq:3.16}
	\end{align}
	where $\tilde{c}=2\max\Big\{ \frac{1}{\alpha},\frac{\beta}{\alpha^{2}}\norm{A}^{2}_{\mathcal{H} \rightarrow \mathcal{H}}\Big\}$.
	These estimations imply the boundedness of the reconstruct formula \eqref{eq:3.14}.
	
	Now we define the operator
		$\mathscr{R}:\mathcal{B}^{s}(\ell^2(\Lambda),\C^{[2]}) \longrightarrow \mathcal{H}$ as
		\begin{align*}
				\mathscr{R}(T)=\sum_{j \in \Lambda}\Big( a_{\frac{r}{N}j} - \sum_{i \in \Lambda}\bar{c_{ij}}\langle \sum_{k \in \Lambda}a_{0k}\tilde{g_{k}}, g_{i} \rangle\Big)\tilde{g_{j}},
		\end{align*}
	where $[2]=\{0,\frac{r}{N}\}$ and $D=[a_{\lambda j}]_{\lambda \in [2],\, j \in \Lambda}$. Equivalently, by the equation \eqref{eq:4.1}, we have
		\begin{align}
			\notag
			\mathscr{R}(T)=&\sum_{j \in \Lambda}\Big( a_{\frac{r}{N}j} - \langle \sum_{k \in \Lambda}a_{0k}\tilde{g_{k}}, A^{*}g_{j} \rangle\Big)\tilde{g_{j}}\\
			= & \sum_{j \in \Lambda}\Big( a_{\frac{r}{N}j} - \langle A\left( \sum_{k \in \Lambda}a_{0k}\tilde{g_{k}}\right), g_{j} \rangle\Big)\tilde{g_{j}}.
			\label{eq:3.17}
			\end{align}
  By using the equation \eqref{eq:3.17}, and the same bounds as in \eqref{eq:3.15} and \eqref{eq:3.16}, we can prove that
	$\mathscr{R}$ is a well-defined bounded operator. Indeed,
	\begin{align*}
		\norm{\mathscr{R}(T)}^{2 }_{\mathcal{H}}=& 	\Big\|\sum_{j \in \Lambda}\Big( a_{\frac{r}{N}j} - \langle A\Big( \sum_{k \in \Lambda}a_{0k}\tilde{g_{k}}\Big), g_{j} \rangle\Big)\tilde{g_{j}}\Big\|^{2 }_{\mathcal{H}}\\
		\leq & \tilde{c}\Big(\sum_{j \in \Lambda}{\abs{\langle a_{\frac{r}{N}j} \rangle}}^{2} +\sum_{j \in \Lambda}{\Big|\langle \sum_{k \in \Lambda}a_{0k}\tilde{g_{k}}, g_{j} \rangle\Big|}^{2} \Big)\\
			\leq & \tilde{c}\Big(\sum_{j \in \Lambda}{\abs{\langle a_{\frac{r}{N}j} \rangle}}^{2} +\sum_{j \in \Lambda}{\abs{\langle a_{0j} \rangle}}^{2} \Big)\\
			= & \tilde{c}{\norm{T}}^{2}_{\mathcal{B}(\ell^2(\Lambda),\C^{[2]})}.
			\end{align*}
			 To justify the last inequality, let
			 $u=\sum_{k \in \Lambda}a_{0k}\tilde{g_{k}} \in \mathcal{H}$.
	Since $\{g_j\}_{j \in \Lambda}$ and $\{\tilde{g_j}\}_{j \in \Lambda}$ are canonical pair of dual frames, we have that
	\begin{align*}
		u=\sum_{j \in \Lambda  }\langle u, g_{j} \rangle \tilde{g_{j}}.
	\end{align*}
	By Lemma \ref{L:2.1}, the coefficients have the least $\ell^{2}(\Lambda)$-norm. In particular,
	\begin{align*}
	\sum_{j \in \Lambda}{\Big|\langle \underbrace{\sum_{k \in \Lambda}a_{0k}\tilde{g_{k}}}_{u} ,g_{j} \rangle\Big|}^{2}=\sum_{j \in \Lambda}\abs{\langle u,g_{j}\rangle}^{2} \leq \sum_{j \in \Lambda}{\abs{ a_{0j}}}^{2}.	
	\end{align*}
	Finally, the operator $\mathscr{R}$ is linear, and by the equation \eqref{eq:3.14}, we have
	\begin{align*}
		\mathscr{R}(\left[\langle x_{\lambda},g_{j}\rangle\right]_{\lambda,\, j \in \Lambda})= & \sum_{j \in \Lambda}\Big( \langle x_{\frac{r}{N}}, g_{j} \rangle - \sum_{i \in \Lambda}\bar{c_{ij}}\langle \sum_{k \in \Lambda} \langle x_{0},g_{k}\rangle\tilde{g_{k}}, g_{i} \rangle\Big)\tilde{g_{j}}\\
		= & \sum_{j \in \Lambda}\Big( \langle x_{\frac{r}{N}}, g_{j} \rangle - \sum_{i \in \Lambda}\bar{c_{ij}}\langle x_{0}, g_{i} \rangle\Big)\tilde{g_{j}}\\
		= & w.
	\end{align*}
	So, it is the reconstruction operator.\\
	
	\textbf{Case-(ii)} When $\lambda \in 2\Z^{+}+\frac{r}{N} \cup \left\{\frac{r}{N}\right\}$, we compute
	\begin{align*}
		\notag
		\langle w, g_{j} \rangle =& 	\langle x_{\lambda+2-\frac{r}{N}} -Ax_{\lambda}, g_{j} \rangle\\
		\notag
		= & \langle x_{\lambda+2-\frac{r}{N}}, g_{j} \rangle - \langle Ax_{\lambda}, g_{j} \rangle\\
		\notag
		= & \langle x_{\lambda+2-\frac{r}{N}}, g_{j} \rangle - \langle x_{\lambda}, A^{*}g_{j} \rangle\\
		\notag
		= & \langle x_{\lambda+2-\frac{r}{N}}, g_{j} \rangle - \langle x_{\lambda}, \sum_{i \in \Lambda}c_{ij}g_{i} \rangle\\
		= & \langle x_{\lambda+2-\frac{r}{N}}, g_{j} \rangle - \sum_{i \in \Lambda}\bar{c_{ij}}\langle x_{\lambda}, g_{i} \rangle.
		\end{align*}
	By using the same procedure as we have done for \textbf{Case-(i)}, we define the reconstruction operator
			$\mathscr{R}:\mathcal{B}^{s}(\ell^2(\Lambda),\C^{[2]}) \longrightarrow \mathcal{H}$ as
		\begin{align*}
		\mathscr{R}(T)=\sum_{j \in \Lambda}\Big( a_{2j} - \sum_{i \in \Lambda}\bar{c_{ij}}\langle \sum_{k \in \Lambda}a_{\frac{r}{N}k}\tilde{g_{k}}, g_{i} \rangle\Big)\tilde{g_{j}},	
		\end{align*}
		where $[2]=\{\frac{r}{N},2\}$ and $D=[a_{\lambda j}]_{\lambda \in [2],\, j \in \Lambda}$. Note that
		\begin{align*}
			\mathscr{R}(\left[\langle x_{\lambda},g_{j}\rangle\right]_{\lambda,\, j \in \Lambda})= & \sum_{j \in \Lambda}\Big( \langle x_{2}, g_{j} \rangle - \sum_{i \in \Lambda}\bar{c_{ij}}\langle \sum_{k \in \Lambda} \langle x_{\frac{r}{N}},g_{k}\rangle\tilde{g_{k}}, g_{i} \rangle\Big)\tilde{g_{j}}\\
			= & \sum_{j \in \Lambda}\Big( \langle x_{2}, g_{j} \rangle - \sum_{i \in \Lambda}\bar{c_{ij}}\langle x_{\frac{r}{N}}, g_{i} \rangle\Big)\tilde{g_{j}}\\
			= & w.
		\end{align*}
		
	\textbf{Case-(iii)} When $\lambda \in 2\Z^{-}+\frac{r}{N} $, we compute
				\begin{align*}
				\notag
				\langle w, g_{j} \rangle =& 	\langle x_{\lambda-2-\frac{r}{N}} -Ax_{\lambda}, g_{j} \rangle\\
				\notag
				= & \langle x_{\lambda-2-\frac{r}{N}}, g_{j} \rangle - \langle Ax_{\lambda}, g_{j} \rangle\\
				\notag
				= & \langle x_{\lambda-2-\frac{r}{N}}, g_{j} \rangle - \langle x_{\lambda}, A^{*}g_{j} \rangle\\
				\notag
				= & \langle x_{\lambda-2-\frac{r}{N}}, g_{j} \rangle - \langle x_{\lambda}, \sum_{i \in \Lambda}c_{ij}g_{i} \rangle\\
				= & \langle x_{\lambda-2-\frac{r}{N}}, g_{j} \rangle - \sum_{i \in \Lambda}\bar{c_{ij}}\langle x_{\lambda}, g_{i} \rangle.
			\end{align*}
			By using the same procedure as we have done for \textbf{Case-(i)}, we define the reconstruction operator
			$\mathscr{R}:\mathcal{B}^{s}(\ell^2(\Lambda),\C^{[2]}) \longrightarrow \mathcal{H}$ as
			\begin{align*}
				\mathscr{R}(T)=\sum_{j \in \Lambda}\Big( a_{-4j} - \sum_{i \in \Lambda}\bar{c_{ij}}\langle \sum_{k \in \Lambda}a_{-2+\frac{r}{N}k}\tilde{g_{k}}, g_{i} \rangle\Big)\tilde{g_{j}},	
			\end{align*}
			where $[2]=\{-2+\frac{r}{N},-4\}$ and $D=[a_{\lambda j}]_{\lambda \in [2],\, j \in \Lambda}$. Then,
			\begin{align*}
				\mathscr{R}(\left[\langle x_{\lambda},g_{j}\rangle\right]_{\lambda,\, j \in \Lambda})= & \sum_{j \in \Lambda}\Big( \langle x_{-4}, g_{j} \rangle - \sum_{i \in \Lambda}\bar{c_{ij}}\langle \sum_{k \in \Lambda} \langle x_{-2+\frac{r}{N}},g_{k}\rangle\tilde{g_{k}}, g_{i} \rangle\Big)\tilde{g_{j}}\\
				= & \sum_{j \in \Lambda}\Big( \langle x_{-4}, g_{j} \rangle - \sum_{i \in \Lambda}\bar{c_{ij}}\langle x_{-2+\frac{r}{N}}, g_{i} \rangle\Big)\tilde{g_{j}}\\
				= & w,
			\end{align*}
	which completes the proof.
			 	\end{proof}
			 	
			 	The following example illustrates the Theorem \ref{Th3.12}.
			
			 \begin{example}
			 	   Let $\mathcal{H}=\ell^2(\Lambda)$ and consider the non-uniform discrete dynamical system \eqref{Dynamical System} with $A(e_{j})=\lambda_{j}e_{j}$ where $\lambda_{j}=	\begin{cases}
			 	   	\frac{1}{2^{\frac{j}{2}}}, & \; \text{if} \; j \in 2\Z^{+}\cup\{0\}; \\
			 	   	{2^{\frac{j}{2}}}, & \; \text{if} \; j \in 2\Z^{-}; \\
			 	   	0, & \; \text{if} \, j \in 2\Z+\frac{r}{N}.
			 	  \end{cases} $

			 	    Let $\{g_j\}_{j \in \Lambda}=\{e_{j}\}_{j \in \Lambda}$ be the standard orthonormal basis for $\ell^2(\Lambda)$, which forms a Parseval frame for $\ell^2(\Lambda)$  with frame bound $1$.
			 	Let $x_{0}=e_{\frac{r}{N}},x_{-2}=e_{-2}$. Then
			 	\begin{align*}
			 	Ae_{j}=A^{*}e_{j}=\sum_{i \in \Lambda}c_{ij}e_{i} \;\; \text{where, each} \;\; c_{ij}=\langle A^{*}e_{j},e_{i}\rangle, \, i \in \Lambda.	
			 	\end{align*}
			 	\textbf{Case-(i)} When $\lambda \in 2\Z$, define $ \mathscr{R}:\mathcal{B}^{s}(\ell^2(\Lambda),\C^{[2]}) \longrightarrow \mathcal{H} \text{ by} $
			 	\begin{align*}
			 			\mathscr{R}(T)=\sum_{j \in \Lambda}\Big( a_{\frac{r}{N}j} - \sum_{i \in \Lambda}\bar{c_{ij}}\langle \sum_{k \in \Lambda}a_{0k}e_{k}, e_{i} \rangle\Big)e_{j}, \text{where}\, [2]=\Big\{0,\frac{r}{N}\Big\}.
			 	\end{align*}\\
			 		\textbf{Case-(ii)} When $\lambda \in 2\Z^{+}+\frac{r}{N} \cup \left\{\frac{r}{N}\right\}$, define
			 		$\mathscr{R}:\mathcal{B}^{s}(\ell^2(\Lambda),\C^{[2]}) \longrightarrow \mathcal{H}\, \text{by} $
			 		\begin{align*}
			 			\mathscr{R}(T)=\sum_{j \in \Lambda}\Big( a_{2j} - \sum_{i \in \Lambda}\bar{c_{ij}}\langle \sum_{k \in \Lambda}a_{\frac{r}{N}k}e_{k}, e_{i} \rangle\Big)e_{j},\text{where}\, [2]=\Big\{\frac{r}{N},2\Big\}.
			 		\end{align*}
			 \textbf{Case-(iii)} When $\lambda \in 2\Z^{-}+\frac{r}{N} $, define
			 		$ \mathscr{R}:\mathcal{B}^{s}(\ell^2(\Lambda),\C^{[2]}) \longrightarrow \mathcal{H} \quad \text{by}$
			 		\begin{align*}
			 			\mathscr{R}(T)=\sum_{j \in \Lambda}\Big( a_{-4j} - \sum_{i \in \Lambda}\bar{c_{ij}}\langle \sum_{k \in \Lambda}a_{-2+\frac{r}{N}k}e_{k}, e_{i} \rangle\Big)e_{j},\text{where}\, [2]=\Big\{-2+\frac{r}{N},-4\Big\}.
			 		\end{align*}
			Then, in each case   $	\norm{\mathscr{R}(T)}^{2 }_{\mathcal{H}} \leq 2{\norm{T}}^{2}_{\mathcal{B}(\ell^2(\Lambda),\C^{[2]})}$, i.e., $\mathscr{R}$ is a bounded linear operator and
			 	$\mathscr{R}(D(x_{0},x_{-2},w))=w$. Hence, $w$ can be recovered stably in finitely many iterations.
			 \end{example}

		In the next result, we constrain the source term to lie within a closed subspace $W \subset \mathcal{H}$. From a practical perspective, this case is particularly significant, despite being more mathematically intricate. For instance, in applications like environmental monitoring, prior knowledge about the locations of primary pollution sources translates to considering closed subspaces of the ambient space
		$\mathcal{H}$. However, the mathematical formulation of a solution in this context proves to be more nuanced. Now we give a necessary condition for the stable recovery of source term of the the non-uniform discrete dynamical system \eqref{Dynamical System} in finitely many iterations.
	\begin{theorem}\label{Th3.14}
Suppose
\begin{enumerate}
\item $\{g_j\}_{j \in \Lambda}$ is a Bessel sequence in $\mathcal{H}$ with Bessel bound $\beta>0$.
\item $W$ is a closed subspace $\mathcal{H}$ and $P_{W}: \mathcal{H} \rightarrow  \mathcal{H} $ denotes the orthogonal projection onto $W$.
\end{enumerate}
Then, for the non-uniform discrete dynamical system \eqref{Dynamical System}, with any arbitrary initial states $x_{0},x_{-2} \in  \mathcal{H}$ and $1 \notin \sigma(A)$, $\{P_{W}(I-{A^*})^{-1}g_{j}\}_{j \in \Lambda}$ is a frame for $W$ if the source term
		$w \in W$ can be stably recovered from the measurements $D(x_{0},x_{-2},w)=[\langle x_{\lambda},g_{j}\rangle]_{\lambda \in [2K],\, j \in \Lambda}$   for some $1 \leq |[2K]| < \infty$.
	\end{theorem}
	
	\begin{proof}
		Suppose that the source term
		$w \in W$ can be stably  recovered from the measurements $D(x_{0},x_{-2},w)=[\langle x_{\lambda},g_{j}\rangle]_{\lambda \in [2K],\,j \in \Lambda}$ in $|[2K]|$ time of iterations. That is, there exists a bounded linear operator
		$\mathscr{R}:\mathcal{B}(\ell^2(\Lambda),\C^{[2K]}) \longrightarrow \mathcal{H}$ such that
		\begin{align*}
			\mathscr{R}(D(x_{0},x_{-2},w))=w \; \text{for all} \; x_{0},x_{-2} \in \mathcal{H}, w \in W.
		\end{align*}
		Hence, there exists a positive constant $c$ such that
		\begin{align}
			& \norm{\mathscr{R}(T)} \leq c\sum_{\lambda \in [2K]}\sum_{j \in \Lambda}\abs{a_{\lambda j}}^{2}, \,\, T \in \mathcal{B}^{s}(\ell^2(\Lambda),\C^{[2K]}),
			\label{eq:3.18}
			\intertext{and}
			 & \mathscr{R}(D(x_{0},x_{-2},w))=w, \, x_{0},x_{-2} \in \mathcal{H},w \in W.
			\label{eq:3.19}
		\end{align}

	Our aim is to prove that $\{P_{W}(I-{A^*})^{-1}g_{j}\}_{j \in \Lambda}$ is a frame for $W$. By using \eqref{eq:3.18} and \eqref{eq:3.19}, we compute
		\begin{align*}
			{\norm{w}}^{2}_{\mathcal{H}} \leq & c\sum_{\lambda \in [2K]}\sum_{j \in \Lambda}\abs{\langle x_{\lambda},g_{j}\rangle}^{2}\\
			=& c \Big[\sum_{j \in \Lambda}\Big\{\sum_{\lambda \in [2K] \cap 2\Z^{+}}\abs{\langle x_{\lambda},g_{j}\rangle}^{2}+\sum_{\lambda \in [2K] \cap 2\Z^{+}+\frac{r}{N} \cup \left\{\frac{r}{N}\right\}}\abs{\langle x_{\lambda},g_{j}\rangle}^{2} \Big. \Big. \\
			&
			\Big. \Big. + \sum_{\lambda \in [2K] \cap 2\Z^{-}}\abs{\langle x_{\lambda},g_{j}\rangle}^{2}+\sum_{\lambda \in [2K] \cap 2\Z^{-}+\frac{r}{N}}\abs{\langle x_{\lambda},g_{j}\rangle}^{2}\Big\}\Big]\\
			=& c \Big[\sum_{j \in \Lambda}\Big\{\sum_{\lambda \in [2K] \cap 2\Z^{+}}\Big|\langle A^{\lambda}x_{0}+\sum_{i=0}^{\lambda-1}A^{i}w,g_{j}\rangle\Big|^{2} \Big. \Big. \\
			& +\sum_{\lambda \in [2K] \cap 2\Z^{+}+\frac{r}{N} \cup \left\{\frac{r}{N}\right\}}\Big|\langle A^{\lambda+1-\frac{r}{N}}x_{0}+\sum_{i=0}^{\lambda-\frac{r}{N}}A^{i}w,g_{j}\rangle\Big|^{2}\\
			& + \sum_{\lambda \in [2K] \cap 2\Z^{-}}\Big|\langle A^{-\lambda-2}x_{-2}+\sum_{i=0}^{-\lambda-3}A^{i}w,g_{j}\rangle\Big|^{2}\\
			& \Big. \Big.+\sum_{\lambda \in [2K] \cap 2\Z^{-}+\frac{r}{N}}\Big|\langle A^{-\lambda-1+\frac{r}{N}}x_{-2}+\sum_{i=0}^{-\lambda-2+\frac{r}{N}}A^{i}w,g_{j}\rangle\Big|^{2}\Big\}\Big]\\
			=&  c\Big[\sum_{j \in \Lambda}\Big\{\sum_{\lambda \in [2K] \cap 2\Z^{+}}\abs{\langle A^{\lambda}x_{0}+(I-A^{\lambda})(I-A)^{-1}w,g_{j}\rangle}^{2}  \Big. \Big. \\
			& +\sum_{\lambda \in [2K] \cap 2\Z^{+}+\frac{r}{N} \cup \left\{\frac{r}{N}\right\}}\abs{\langle A^{\lambda+1-\frac{r}{N}}x_{0}+(I-A^{\lambda+1-\frac{r}{N}})(I-A)^{-1}w,g_{j}\rangle}^{2} \\
			& +	\sum_{\lambda \in [2K] \cap 2\Z^{-}}\abs{\langle A^{-\lambda-2}x_{-2}+(I-A^{-\lambda-2})(I-A)^{-1}w,g_{j}\rangle}^{2} \\
			& \Big. \Big. + \sum_{\lambda \in [2K] \cap 2\Z^{-}+\frac{r}{N}}\abs{\langle A^{-\lambda-1+\frac{r}{N}}x_{-2}+(I-A^{-\lambda-1+\frac{r}{N}})(I-A)^{-1}w,g_{j}\rangle}^{2}\Big\}\Big],
		\end{align*}
		where we have used that the non-uniform discrete dynamical system \eqref{Dynamical System} can be expressed as
		\[
		x_{\lambda}=
		\begin{cases}
			A^{\lambda}x_{0}+(I-A^{\lambda})(I-A)^{-1}w,\,\,\lambda \in  2\Z^{+};&\\
			A^{\lambda+1-\frac{r}{N}}x_{0}+(I-A^{\lambda+1-\frac{r}{N}})(I-A)^{-1}w,\,\, \lambda \in  2\Z^{+}+\frac{r}{N} \cup \left\{\frac{r}{N}\right\};&\\
			A^{-\lambda-2}x_{-2}+(I-A^{-\lambda-2})(I-A)^{-1}w,\,\, \lambda \in  2\Z^{-};&\\
			A^{-\lambda-1+\frac{r}{N}}x_{-2}+(I-A^{-\lambda-1+\frac{r}{N}})(I-A)^{-1}w,\,\, \lambda \in  2\Z^{-}+\frac{r}{N}.
		\end{cases}
		\]
		For any $w \in W$, consider $x_{0}=x_{-2}=(I-A)^{-1}w$ and by putting it into the above inequality, we get
		\begin{align*}
			\norm{w}^{2}_{\mathcal{H}}  \leq  |[2K]|c\sum_{j \in \Lambda}\abs{\langle (I-A)^{-1}w,g_{j}\rangle}^{2}
			\leq c'\norm{(I-A)^{-1}w}^{2}
			\leq c''\norm{w}^{2}_{\mathcal{H}},
		\end{align*}
		where $c'= |[2K]| c \beta$ and $c''=|[2K]|c \beta\norm{(I-A)^{-1}}^{2}_{\mathcal{H} \rightarrow \mathcal{H}}$.
		Note that
		\begin{align*}
			\sum_{j \in \Lambda}\abs{\langle (I-A)^{-1}w,g_{j}\rangle}^{2} =\sum_{j \in \Lambda}\abs{\langle w,((I-A)^{-1})^{*}g_{j}\rangle}^{2}
			=\sum_{j \in \Lambda}\abs{\langle w,P_{W}(I-A^*)^{-1}g_{j}\rangle}^{2}.
		\end{align*}
		So, we have
		\begin{align*}
			\frac{1}{|[2K]|c} \norm{w}^{2}_{\mathcal{H}}  \leq \sum_{j \in \Lambda}\abs{\langle w,P_{W}(I-A^*)^{-1}g_{j}\rangle}^{2}
			 \leq c'' \norm{w}^{2}_{\mathcal{H}},\, w \in W.
		\end{align*}
		Therefore, $\{P_{W}(I-{A^*})^{-1}g_{j}\}_{j \in \Lambda}$ is a frame for $W$.
		
	\end{proof}
	
		Now, we present an example which ratifies that the recovery of source $w \in W \subset \mathcal{H}$ is not necessarily possible in finite number of iterations even if $\{P_{W}(I-{A^*})^{-1}g_{j}\}_{j \in \Lambda}$ is a frame for $W$.
	
		\begin{example}
		Let $\mathcal{H}=\ell^2(\Lambda)$ 	and $ W $ is the one-dimensional subspace of $\ell^2(\Lambda)$ generated by $w$, where
		\begin{align}
			\notag
			w = &(\ldots,w_{-4},w_{-4+\frac{r}{N}},w_{-2},w_{-2+\frac{r}{N}},w_{0},w_{\frac{r}{N}},w_{2},w_{2+\frac{r}{N}},\ldots)\\
			= &(\ldots,\frac{-1}{2^2},\frac{-1}{3^2},\frac{-1}{2},\frac{-1}
			{3},1,\frac{1}{3},\frac{1}{2},\frac{1}{3^2},\ldots).
			\label{eq:3.20}
		\end{align}
			Let
		\[
		\mathbf{A} =\diag{\cdots,\lambda_{-4},\lambda_{-4+\frac{r}{N}},\lambda_{-2}, \lambda_{-2+\frac{r}{N}},\lambda_{0},\lambda_{\frac{r}{N}}, \lambda_{2},\cdots},
		\]
		where $0 < \lambda_{i} <1, i \in \Lambda$ and $\lambda_{i} \neq \lambda_{j} \; \text{if} \; i \neq j$.
		Then, $A$ acts as a bounded linear operator on $\ell^2(\Lambda)$.
		
		Let $g=(I-A)w$. Then, $P_{W}(I-A^{*})^{-1}g=w$. Therefore, $\{P_{W}(I-{A^*})^{-1}g\}$ is a frame for $W$. Moreover, by using \eqref{eq:3.20} and that $A$ is diagonal operator, we have
		\begin{align}		
			\notag
			& g \\
			\notag
			 =&(\ldots, g_{-4+\frac{r}{N}},g_{-2},g_{-2+\frac{r}{N}},g_{0},g_{\frac{r}{N}},g_{2},g_{2+\frac{r}{N}},\ldots)\\
		 =  & \Big(\ldots,-\frac{(1-\lambda_{-4+\frac{r}{N}})}{3^2},-\frac{(1-\lambda_{-2})}{2},-\frac{(1-\lambda_{-2+\frac{r}{N}})}
			{3},(1-\lambda_{0}),\frac{(1-\lambda_{\frac{r}{N}})}{3},\frac{(1-\lambda_{2})}{2},\frac{(1-\lambda_{2+\frac{r}{N}})}{3^2},\ldots\Big),
			\label{eq:3.21}
		\end{align}
	where each coordinate $g_{i}$ is non-zero.
	
		For some $c \neq 0$, consider the non-uniform discrete dynamical system
		\begin{align*}
			x_{\lambda+\frac{r}{N}}=& Ax_{\lambda} +cw,\,\, \lambda \in 2\Z;  \\
			x_{\lambda+2-\frac{r}{N}}=& Ax_{\lambda} +cw,\,\, \lambda \in 2\Z^{+}+\frac{r}{N} \cup \left\{\frac{r}{N}\right\};\\
			x_{\lambda-2-\frac{r}{N}} =& Ax_{\lambda} +cw,\,\, \lambda \in 2\Z^{-}+\frac{r}{N}.
		\end{align*}
		Then,
		\begin{align}
			\langle x_{\lambda},g \rangle =
			\begin{cases}
				\langle x_{0},A^{\lambda}g \rangle + c\langle w, K_{\lambda}g \rangle,\,\, \lambda \in  2\Z^{+};&\\
				\langle x_{0},	A^{\lambda+1-\frac{r}{N}}g \rangle + c\langle w, K_{\lambda}g \rangle,\,\,\lambda \in  2\Z^{+}+\frac{r}{N} \cup \left\{\frac{r}{N}\right\};&\\
				\langle x_{-2},	A^{-\lambda-2}g \rangle + c\langle w, K_{\lambda}g \rangle,\,\,\lambda \in  2\Z^{-};&\\
				\langle x_{-2},	A^{-\lambda-1+\frac{r}{N}}g \rangle + c\langle w, K_{\lambda}g \rangle,\,\,\lambda \in  2\Z^{-}+\frac{r}{N},
			\end{cases}
			\label{eq:3.22}
		\end{align}
		where
		\[
		K_{\lambda}=
		\begin{cases}
			I+A+A^{2}+\cdots+A^{\lambda-1},\,\, \lambda \in  2\Z^{+};&\\
			I+A+A^{2}+\cdots+A^{\lambda-\frac{r}{N}},\,\,\lambda \in  2\Z^{+}+\frac{r}{N} \cup \left\{\frac{r}{N}\right\};&\\
			I+A+A^{2}+\cdots+A^{-\lambda-3},\,\, \lambda \in  2\Z^{-};&\\
			I+A+A^{2}+\cdots+A^{-\lambda-2+\frac{r}{N}},\,\,\lambda \in  2\Z^{-}+\frac{r}{N}.
		\end{cases}
		\]
		Let
		\begin{align}
			\notag
			x_{0} =&(\ldots,0,0,a_{0},a_{\frac{r}{N}},\ldots,a_{2K-2},a_{2K-2+\frac{r}{N}},0,0,\ldots),
		\intertext{and}
			x_{-2} =&(\ldots,0,0,d_{-2K},d_{-2K+\frac{r}{N}},\ldots,d_{-2},d_{-2+\frac{r}{N}},0,0,\ldots).
			\label{eq:3.23}
		\end{align}
		Using \eqref{eq:3.20} and \eqref{eq:3.21} in terms of the coordinates of the vectors, for
	$\lambda=0,\frac{r}{N},2,2+\frac{r}{N},\ldots, \break 2K-2,2K-2+\frac{r}{N}$, the system \eqref{eq:3.22} can be expressed as
		\[
		\begin{bmatrix}
			\langle x_{0},g \rangle \\
			\langle x_{\frac{r}{N}},g \rangle \\
			\langle x_{2},g \rangle \\
			\vdots \\
			\langle x_{2K-2},g \rangle \\
			\langle x_{2K-2+\frac{r}{N}},g \rangle
		\end{bmatrix} =
		\begin{bmatrix}
			g_{0} & g_{\frac{r}{N}} & \cdots & g_{2K-2+\frac{r}{N}} & 0 \\
			\lambda_{0}g_{0} & \lambda_{\frac{r}{N}}g_{\frac{r}{N}} & \cdots & \lambda_{2K-2+\frac{r}{N}}g_{2K-2+\frac{r}{N}} & b_{\frac{r}{N}} \\
			\lambda^{2}_{0}g_{0} & \lambda^{2}_{\frac{r}{N}}g_{\frac{r}{N}} & \cdots  & \lambda^{2}_{2K-2+\frac{r}{N}}g_{2K-2+\frac{r}{N}} & b_{2} \\
			\vdots & \vdots & \cdots & \vdots & \vdots \\
			\lambda^{2K-2}_{0}g_{0} & \lambda^{2K-2}_{\frac{r}{N}}g_{\frac{r}{N}} & \cdots & \lambda^{2K-2}_{2K-2+\frac{r}{N}}g_{2K-2+\frac{r}{N}} & b_{2K-2} \\
			\lambda^{2K-1}_{0}g_{0} & \lambda^{2K-1}_{\frac{r}{N}}g_{\frac{r}{N}} & \cdots & \lambda^{2K-1}_{2K-2+\frac{r}{N}}g_{2K-2+\frac{r}{N}} & b_{2K-2+\frac{r}{N}} \\
		\end{bmatrix}
		\begin{bmatrix}
			a_{0} \\
			a_{\frac{r}{N}} \\
			\vdots \\
			a_{2K-2+\frac{r}{N}} \\
			c
		\end{bmatrix},
		\]
		where $ b_{i}= \langle w, K_{i}g \rangle, \,i \in \left\{0,\frac{r}{N},2,2+\frac{r}{N},\ldots,2K-2,2K-2+\frac{r}{N}\right\} $. Notice that the matrix is of dimension $2K \times (2K+1)$. Moreover, $g_{i} \neq 0, \, i \in \left\{0,\frac{r}{N},2,2+\frac{r}{N},\ldots,2K-2,2K-2+\frac{r}{N}\right\} $,
		we have
		\[
		\begin{vmatrix}
			g_{0} & g_{\frac{r}{N}} & \cdots & g_{2K-2} & g_{2K-2+\frac{r}{N}} \\
			\lambda_{0}g_{0} & \lambda_{\frac{r}{N}}g_{\frac{r}{N}} & \cdots & \lambda_{2K-2}g_{2K-2} & \lambda_{2K-2+\frac{r}{N}}g_{2K-2+\frac{r}{N}} \\
			\lambda^{2}_{0}g_{0} & \lambda^{2}_{\frac{r}{N}}g_{\frac{r}{N}} & \cdots & \lambda^{2}_{2K-2}g_{2K-2} & \lambda^{2}_{2K-2+\frac{r}{N}}g_{2K-2+\frac{r}{N}}  \\
			\vdots & \vdots & \cdots & \vdots & \vdots  \\
			\lambda^{2K-2}_{0}g_{0} & \lambda^{2K-2}_{\frac{r}{N}}g_{\frac{r}{N}} & \cdots & \lambda^{2K-2}_{2K-2}g_{2K-2} & \lambda^{2K-2}_{2K-2+\frac{r}{N}}g_{2K-2+\frac{r}{N}}  \\
			\lambda^{2K-1}_{0}g_{0} & \lambda^{2K-1}_{\frac{r}{N}}g_{\frac{r}{N}} & \cdots & \lambda^{2K-1}_{2K-2}g_{2K-2} & \lambda^{2K-1}_{2K-2+\frac{r}{N}}g_{2K-2+\frac{r}{N}}  \\
		\end{vmatrix} \neq 0.
		\]
		Therefore, the columns in the above $2K \times 2K$ matrix are linearly independent.
		Now, for $\lambda = -2K,-2K+\frac{r}{N},\ldots,-2, -2+\frac{r}{N} $, the system \eqref{eq:3.22} can be expressed as
		\[
		\begin{bmatrix}
			\langle x_{-2K},g \rangle \\
			\langle x_{-2K+\frac{r}{N}},g \rangle \\
			\vdots \\
			\langle x_{-2},g \rangle \\
			\langle x_{-2+\frac{r}{N}},g \rangle
		\end{bmatrix} =
		\begin{bmatrix}
			\lambda^{2K-1}_{-2K}g_{-2K} & \lambda^{2K-1}_{-2K+\frac{r}{N}}g_{-2K+\frac{r}{N}} & \cdots & \lambda^{2K-1}_{-2+\frac{r}{N}}g_{-2+\frac{r}{N}} & b_{-2K} \\
			\lambda^{2K-2}_{-2K}g_{-2K} & \lambda^{2K-2}_{-2K+\frac{r}{N}}g_{-2K+\frac{r}{N}} & \cdots & \lambda^{2K-2}_{-2+\frac{r}{N}}g_{-2+\frac{r}{N}} & b_{-2K+\frac{r}{N}} \\
			\vdots & \vdots & \cdots  & \vdots & \vdots \\
			g_{-2K} & g_{-2K+\frac{r}{N}} & \cdots  & g_{-2+\frac{r}{N}} & 0 \\
			\lambda_{-2K}g_{-2K} & \lambda_{-2K+\frac{r}{N}}g_{-2K+\frac{r}{N}} & \cdots  & \lambda_{-2+\frac{r}{N}}g_{-2+\frac{r}{N}} & b_{-2+\frac{r}{N}}
		\end{bmatrix}
		\begin{bmatrix}
			d_{-2K}\\
			d_{-2K+\frac{r}{N}} \\
			\vdots \\
			d_{-2+\frac{r}{N}} \\
			c
		\end{bmatrix},
		\]
		where $ b_{i}= \langle w, K_{i}g \rangle, \,i \in \left\{-2K,-2K+\frac{r}{N},\ldots,-2, -2+\frac{r}{N}\right\} $. Again, the above matrix is of dimension $2K \times (2K+1)$ and since each $g_{i} \neq 0 ,\, i \in \left\{-2K,-2K+\frac{r}{N},\ldots,-2, -2+\frac{r}{N}\right\} $,
		we have
		\[
		\begin{vmatrix}
			\lambda^{2K-1}_{-2K}g_{-2K} & \lambda^{2K-1}_{-2K+\frac{r}{N}}g_{-2K+\frac{r}{N}} & \cdots & \lambda^{2K-1}_{-2}g_{-2} & \lambda^{2K-1}_{-2+\frac{r}{N}}g_{-2+\frac{r}{N}}  \\
			\lambda^{2K-2}_{-2K}g_{-2K} & \lambda^{2K-2}_{-2K+\frac{r}{N}}g_{-2K+\frac{r}{N}} & \cdots & \lambda^{2K-2}_{-2}g_{-2} & \lambda^{2K-2}_{-2+\frac{r}{N}}g_{-2+\frac{r}{N}} \\
			\vdots & \vdots & \cdots & \vdots & \vdots  \\
			g_{-2K} & g_{-2K+\frac{r}{N}} & \cdots & g_{-2} & g_{-2+\frac{r}{N}}  \\
			\lambda_{-2K}g_{-2K} & \lambda_{-2K+\frac{r}{N}}g_{-2K+\frac{r}{N}} & \cdots & \lambda_{-2}g_{-2} & \lambda_{-2+\frac{r}{N}}g_{-2+\frac{r}{N}}
		\end{vmatrix} \neq 0.
		\]
		Therefore, the columns in the above $2K \times 2K$ matrix are also linearly independent. Thus for any $c \neq 0$, there exists a unique solution $x_{0}$ and $x_{-2}$ of the form \eqref{eq:3.23} such that
		\[
		\begin{bmatrix}
			g_{0} & g_{\frac{r}{N}} & \cdots & g_{2K-2} & g_{2K-2+\frac{r}{N}} & 0 \\
			\lambda_{0}g_{0} & \lambda_{\frac{r}{N}}g_{\frac{r}{N}} & \cdots & \lambda_{2K-2}g_{2K-2} & \lambda_{2K-2+\frac{r}{N}}g_{2K-2+\frac{r}{N}} & b_{\frac{r}{N}} \\
			\lambda^{2}_{0}g_{0} & \lambda^{2}_{\frac{r}{N}}g_{\frac{r}{N}} & \cdots & \lambda^{2}_{2K-2}g_{2K-2} & \lambda^{2}_{2K-2+\frac{r}{N}}g_{2K-2+\frac{r}{N}} & b_{2} \\
			\vdots & \vdots & \cdots & \vdots & \vdots & \vdots \\
			\lambda^{2K-2}_{0}g_{0} & \lambda^{2K-2}_{\frac{r}{N}}g_{\frac{r}{N}} & \cdots & \lambda^{2K-2}_{2K-2}g_{2K-2} & \lambda^{2K-2}_{2K-2+\frac{r}{N}}g_{2N-2+\frac{r}{N}} & b_{2K-2} \\
			\lambda^{2K-1}_{0}g_{0} & \lambda^{2K-1}_{\frac{r}{N}}g_{\frac{r}{N}} & \cdots & \lambda^{2K-1}_{2K-2}g_{2K-2} & \lambda^{2K-1}_{2K-2+\frac{r}{N}}g_{2K-2+\frac{r}{N}} & b_{2K-2+\frac{r}{N}} \\
		\end{bmatrix}
		\begin{bmatrix}
			a_{0} \\
			a_{\frac{r}{N}} \\
			a_{2} \\
			\vdots \\
			a_{2K-2+\frac{r}{N}} \\
			c
		\end{bmatrix} =
		\begin{bmatrix}
			0 \\
			0\\
			0 \\
			\vdots \\
			0 \\
			0
		\end{bmatrix},
		\]
		and
		\[
		\begin{bmatrix}
			\lambda^{2K-1}_{-2K}g_{-2K} & \lambda^{2K-1}_{-2K+\frac{r}{N}}g_{-2K+\frac{r}{N}} & \cdots & \lambda^{2K-1}_{-2}g_{-2} & \lambda^{2K-1}_{-2+\frac{r}{N}}g_{-2+\frac{r}{N}} & b_{-2K} \\
			\lambda^{2K-2}_{-2K}g_{-2K} & \lambda^{2K-2}_{-2K+\frac{r}{N}}g_{-2K+\frac{r}{N}} & \cdots & \lambda^{2K-2}_{-2}g_{-2} & \lambda^{2K-2}_{-2+\frac{r}{N}}g_{-2+\frac{r}{N}} & b_{-2K+\frac{r}{N}} \\
			\vdots & \vdots & \cdots & \vdots & \vdots & \vdots \\
			g_{-2K} & g_{-2K+\frac{r}{N}} & \cdots & g_{-2} & g_{-2+\frac{r}{N}} & 0 \\
			\lambda_{-2K}g_{-2K} & \lambda_{-2K+\frac{r}{N}}g_{-2K+\frac{r}{N}} & \cdots & \lambda_{-2}g_{-2} & \lambda_{-2+\frac{r}{N}}g_{-2+\frac{r}{N}} & b_{-2+\frac{r}{N}}
		\end{bmatrix}
		\begin{bmatrix}
			d_{-2K} \\
			d_{-2K+\frac{r}{N}} \\
			\vdots \\
			d_{-2+\frac{r}{N}} \\
			c
		\end{bmatrix} =
		\begin{bmatrix}
			0 \\
			0 \\
			\vdots \\
			0\\
			0
		\end{bmatrix}.
		\]
		Therefore, with the above choice of $x_{0}$ and $x_{-2}$, we necessarily have
		\begin{align*}
			\langle x_{-2K},g \rangle =
			\langle x_{-2K+\frac{r}{N}},g \rangle =
			\cdots =
			\langle x_{-2+\frac{r}{N}},g \rangle= 	\langle x_{0},g \rangle =
			\langle x_{\frac{r}{N}},g \rangle =
			\cdots =
			\langle x_{2K-2+\frac{r}{N}},g \rangle =0.
		\end{align*}
		Hence, it is not possible to stably recover the source term $cw$ if we only have the first $|[2K]|$ samples of measurements.
		
	\end{example}
	
		\subsection{Infinite Time Iterations}\label{Sec.V}
		In this subsection, we investigate the case in which the data matrix $D(x_{0},x_{-2},w)$ entails infinitely many time samples and $W \subsetneq \mathcal{H}$. In this case, all the data measurements are carried out in the space $\mathcal{B}^{s}(\ell^2(\Lambda),\ell^\infty(\Lambda))$.
		
	We commence this section by establishing the following lemma which affirms that the data matrix for infinite time iterations belongs to the space $\mathcal{B}^{s}(\ell^2(\Lambda),\ell^\infty(\Lambda))$.
		\begin{lemma}\label{L3.16}
			Let $\{g_j\}_{j \in \Lambda} \subset \mathcal{H}$ be a Bessel sequence with Bessel bound $\beta>0$.
		Then, for the non-uniform discrete dynamical system \eqref{Generalized Dynamical System} $(\mathcal{H},W,\mathscr{F},S)$, with any initial states $x_{0},x_{-2} \in  \mathcal{H}$, the data matrix $D(x_{0},x_{-2},w)=[\langle x_{\lambda},g_{j}\rangle]_{\lambda,\,j \in \Lambda} \in \mathcal{B}^{s}(\ell^2(\Lambda),\ell^\infty(\Lambda))$.
	\end{lemma}
	
	\begin{proof}
		By supposition, the states $\{x_{\lambda}\}$ of the dynamical system \eqref{Generalized Dynamical System} satisfy
		\begin{align}
			\norm{x_{\lambda}-S(w)}_{\mathcal{H}} \longrightarrow 0 \; \text{as} \; \abs{\lambda} \longrightarrow \infty.
			\label{eq:3.24}
		\end{align}
		The $\lambda$-th row of $D(x_{0},x_{-2},w)$ is
			$ r_{\lambda}=(\langle x_{\lambda},g_{j}\rangle)_{\lambda,\,j \in \Lambda} $.
		Let
			$	r=(\langle S(w),g_{j}\rangle)_{\lambda,\,j \in \Lambda}$.
	Since
		$\{g_j\}_{j \in \Lambda} \subset \mathcal{H}$  is a Bessel sequence with Bessel bound $\beta>0$.
		By using the equation \eqref{eq:3.24}, we have
		\begin{align}
			\notag	{\norm{r_{\lambda}-r}}^{2}_{\ell^2(\Lambda)}=&\sum_{j \in \Lambda}\abs{\langle x_{\lambda},g_{j}\rangle -\langle S(w),g_{j}\rangle}^{2}\\ \notag
			=&\sum_{j \in \Lambda}\abs{\langle x_{\lambda}-S(w),g_{j}\rangle}^{2}\\
			\label{eq:3.25}
			\leq & \beta	{\norm{x_{\lambda}-S(w)}}^{2}_{\mathcal{H}} \\
			\notag
			\longrightarrow & \, 0 \,\, \text{as} \,\, \abs{\lambda} \rightarrow \infty.
		\end{align}
		This implies,
		$	r_{\lambda} \longrightarrow r \,\, \text{as} \,\, \abs{\lambda} \rightarrow \infty $.
		Hence, by using Lemma \ref{L3.6}, we have
		\begin{align*}
		D(x_{0},x_{-2},w) \in \mathcal{B}^{s}(\ell^2(\Lambda),\ell^\infty(\Lambda)).
		\end{align*}
	\end{proof}
	
	Motivated by \cite[Theorem 3.4]{CFSRDS}, the next theorem betokens the characterization for the stable recovery of the source term of the non-uniform discrete dynamical system \eqref{Generalized Dynamical System} from the data measurements $D(x_{0},x_{-2},w)$.
	
	\begin{theorem}\label{Th3.17}
Suppose
\begin{enumerate}
  \item  $\{g_j\}_{j \in \Lambda} \subset \mathcal{H} $ is a Bessel sequence with  Bessel bound $\beta>0$.
  \item  $W$ is  a closed subspace of $\mathcal{H}$.
\end{enumerate}
		Then, for the non-uniform discrete dynamical system \eqref{Generalized Dynamical System} $(\mathcal{H},W,\mathscr{F},S)$ with the assumption that $\mathscr{F}$ is linear, each  source term
		$w \in W$ can be stably recovered from the measurements $D(x_{0},x_{-2},w)=[\langle x_{\lambda},g_{j}\rangle]_{\lambda,\,j \in \Lambda}$ if and only if $\{S^*g_{j}\}_{j \in \Lambda}$ is a frame for $W$.
	\end{theorem}
	
	\begin{proof}
		For any $x_{0},x_{-2} \in \mathcal{H}$ and any source term $w \in W$,
		by using Lemma \ref{L3.16}, the data matrix
		$D(x_{0},x_{-2},w) \in \mathcal{B}^{s}(\ell^2(\Lambda),\ell^\infty(\Lambda))$. First, assume that each  source term
		$w \in W$ can be stably recovered from the measurements $D(x_{0},x_{-2},w)=[\langle x_{\lambda},g_{j}\rangle]_{\lambda,\,j \in \Lambda}$. That is, there exits a bounded linear operator
		$\mathscr{R}:\mathcal{B}^{s}(\ell^2(\Lambda),\ell^\infty(\Lambda)) \longrightarrow \mathcal{H}$
		such that
		\begin{align*}
				\mathscr{R}(D(x_{0},x_{-2},w))=w \; \text{for all} \, x_{0},x_{-2} \in \mathcal{H}, w \in W.
		\end{align*}
	Hence,
		\begin{align}
			\norm{w}_{\mathcal{H}} \leq \norm{\mathscr{R}}_{\mathcal{B}^{s}(\ell^2(\Lambda),\ell^\infty(\Lambda)) \longrightarrow \mathcal{H}} \sup_{\lambda \in \Lambda}{\Big(\sum_{j \in \Lambda}\abs{\langle x_{\lambda},g_{j}\rangle}^{2}\Big)}^{\frac{1}{2}}, \, x_{0},x_{-2} \in \mathcal{H}, w \in W.
			\label{eq:3.26}
		\end{align}
		Since $x_{0},x_{-2} \in \mathcal{H}$ are arbitary, so we can take $x_{0}=x_{-2}=S(w)$.
		Hence, $x_{\lambda}=S(w) \; \text{for all} \; \lambda \in \Lambda$ and the equation \eqref{eq:3.26} becomes
		\begin{align*}
			\norm{w}_{\mathcal{H}} \leq \norm{\mathscr{R}}_{\mathcal{B}^{s}(\ell^2(\Lambda),\ell^\infty(\Lambda)) \longrightarrow \mathcal{H}} {\Big(\sum_{j \in \Lambda}\abs{\langle S(w),g_{j}\rangle}^{2}\Big)}^{\frac{1}{2}}, \; w \in W,
		\end{align*}
	which can be rewritten as
		\begin{align*}
			{\norm{\mathscr{R}}}^{-2}_{\mathcal{B}^{s}(\ell^2(\Lambda),\ell^\infty(\Lambda)) \longrightarrow \mathcal{H}}{\norm{w}}^{2}_{\mathcal{H}}
			\leq &
			\sum_{j \in \Lambda}\abs{\langle w,S^{*}g_{j}\rangle}^{2}, \; w \in W.
			\end{align*}
		Finally, since $\{g_j\}_{j \in \Lambda}$ is  Bessel sequence  with Bessel bound $\beta>0$, we have
		\begin{align*}
		\sum_{j \in \Lambda}\abs{\langle w,S^{*}g_{j}\rangle}^{2} =\sum_{j \in \Lambda}\abs{\langle S(w),g_{j}\rangle}^{2} \leq \beta {\norm{S(w)}}^{2} \leq \beta{\norm{S}}^{2}_{W \longrightarrow \mathcal{H}} {\norm{w}}^{2}_{\mathcal{H}} \; \text{for all} \; w \in W.
		\end{align*}
		Hence,
		\begin{align*}
		A{\norm{w}}^{2}_{\mathcal{H}} \leq  \sum_{j \in \Lambda}\abs{\langle w,S^{*}g_{j}\rangle}^{2} \leq B{\norm{w}}^{2}_{\mathcal{H}} \; \text{for all} \; w \in W,
		\end{align*}
		with $A=	{\norm{\mathscr{R}}}^{-2}_{\mathcal{B}^{s}(\ell^2(\Lambda),\ell^\infty(\Lambda)) \longrightarrow \mathcal{H}}$ and $B=\beta{\norm{S}}^{2}_{W \longrightarrow \mathcal{H}}$.
		Therefore, $\{S^*g_{j}\}_{j \in \Lambda}$ is a frame for $W$.
		
		Conversely, let us assume that $\{S^*g_{j}\}_{j \in \Lambda}$ is a frame for $W$, and $\{\widetilde{g_j}\}_{j \in \Lambda} \subset W$  represents a dual frame of $\{S^*g_{j}\}_{j \in \Lambda}$. Then, by the definition of $D(x_{0},x_{-2},w)$, we have
		\begin{align*}
		(D(x_{0},x_{-2},w)\{\widetilde{g_j}\}_{j \in \Lambda})_{\lambda}= \sum_{j \in \Lambda} \langle x_{\lambda},g_{j} \rangle \widetilde{g_j}, \;\lambda \in \Lambda.
		\end{align*}
		Clearly,
		$(D(x_{0},x_{-2},w)\{\widetilde{g_j}\}_{j \in \Lambda})_{\lambda} \in W$. Now, by using the duality of $\{\widetilde{g_j}\}_{j \in \Lambda} \subset W$, we have
		\begin{align*}
				w=\sum_{j \in \Lambda} \langle w,S^{*}g_{j} \rangle \widetilde{g_j}.
			\end{align*}
	Hence, by using the equation \eqref{eq:3.25}, we compute
		\begin{eqnarray*}
			{\norm{(D\{\widetilde{g_j}\}_{j \in \Lambda})_{\lambda}-w}}^{2}_{\mathcal{H}} &=&	{\norm{\sum_{j \in \Lambda} \langle x_{\lambda},g_{j} \rangle \widetilde{g_j}-\sum_{j \in \Lambda} \langle w,S^{*}g_{j} \rangle \widetilde{g_j}}}^{2}_{\mathcal{H}}\\
			&=&{\norm{\sum_{j \in \Lambda}\Big(\langle x_{\lambda},g_{j} \rangle - \langle w,S^{*}g_{j} \rangle\Big) \widetilde{g_j}}}^{2}_{\mathcal{H}}\\
			& \leq & \beta'\sum_{j \in \Lambda}{\abs{\langle x_{\lambda},g_{j} \rangle - \langle w,S^{*}g_{j} \rangle}}^{2}\\
			&=&\beta'\sum_{j \in \Lambda}{\abs{\langle x_{\lambda}-S(w),g_{j} \rangle}}^{2}\\
			& \leq & \beta'\beta	{\norm{x_{\lambda}-S(w)}}^{2}_{\mathcal{H}} \longrightarrow 0 \; \text{as} \; \abs{\lambda} \rightarrow \infty,
		\end{eqnarray*}
		where $\beta'$ is the upper frame bound of $\{\widetilde{g_j}\}_{j \in \Lambda}$. This implies that $\mathscr{R}:\mathcal{B}^{s}(\ell^2(\Lambda),\ell^\infty(\Lambda)) \longrightarrow \mathcal{H} \quad \text{defined by} \quad
		\mathscr{R}(T)=\lim{T\{\widetilde{g_j}\}_{j \in \Lambda}} \quad \text{satisfies} \quad  \mathscr{R}(D(x_{0},x_{-2},w))=w$.
		Therefore, the conclusion follows from the Theorem \ref{Th3.8}.
	\end{proof}
	
	We now illustrate Theorem \ref{Th3.17} with the following example.
	
	\begin{example}
		 Let $\mathcal{H}=\ell^2(\Lambda)$ and $W=\text{span}\{e_{j}:j \in \Lambda,j \neq 0,-2\}$. Let $\{g_j\}_{j \in \Lambda}=\{e_{j}\}_{j \in \Lambda}$ be the standard orthonormal basis for $\ell^2(\Lambda)$ which form a Bessel sequence with optimal Bessel bound $1$. Let $A(e_{j})=	\begin{cases}
		 	e_{j}, & \; \text{if} \;j \neq 0, -2;\\
		 	2e_{j}, & \; \text{if} \; j \in \Lambda \setminus\{0,-2\}.
		 \end{cases}$ Then, $\rho(A) = 2$.
		Let $x_{0}=x_{-2}=(I-A)w=-w$. Define
	$S:W \longrightarrow \mathcal{H}$ by
		$S(w)=(I-A)w=-w$. Then, all the conditions of the non-uniform discrete dynamical system \eqref{Generalized Dynamical System} are satisfied.
		Also,
		\begin{align*}
			S^{*}(g_{j})=
			\begin{cases}
				0, & \; \text{if} \; j \neq 0, -2; \\
				-e_{j}, & \text{otherwise}.
			\end{cases}
		\end{align*}
		Thus, $\{S^{*}g_{j}\}_{j \in \Lambda}$ is a frame for $W$. Let $\{\tilde{g_{j}}\}_{j \in \Lambda}=	\begin{cases}
			0, & \; \text{if} \, j \neq 0, -2; \\
			-e_{j}, & \text{otherwise},
		\end{cases}$ be the canonical dual frame of $\{S^{*}g_{j}\}_{j \in \Lambda}$. Define  $\mathscr{R}:\mathcal{B}^{s}(\ell^2(\Lambda),\ell^\infty(\Lambda)) \longrightarrow \mathcal{H}$ by  $\mathscr{R}(T)=\lim{T\{\tilde{g_{j}}\}_{j \in \Lambda}}$.
		Then $\mathscr{R}$ is a bounded linear operator such that
		\begin{align*}
			\mathscr{R}(D)=&\lim{D\{\tilde{g_{j}}\}_{j \in \Lambda}}\\
			=& \lim_{\abs{\lambda} \rightarrow \infty}\sum_{j \in \Lambda}\langle x_{\lambda},g_{j}\rangle \tilde{g_{j}}\\
			=& \lim_{\abs{\lambda} \rightarrow \infty}\sum_{j \in \Lambda,j \neq 0,-2}\langle x_{\lambda},e_{j}\rangle (-e_{j})\\
			=& \sum_{j \in \Lambda,j \neq 0,-2}\langle \lim_{\abs{\lambda} \rightarrow \infty}x_{\lambda},e_{j}\rangle (-e_{j})\\
			=& \sum_{j \in \Lambda,j \neq 0,-2}\langle S(w),e_{j}\rangle (-e_{j})\\
			=& \sum_{j \in \Lambda,j \neq 0,-2}\langle -w,e_{j}\rangle (-e_{j})\\
			=& \sum_{j \in \Lambda,j \neq 0,-2}\langle w,e_{j}\rangle e_{j}\\
			=&w.
		\end{align*}
		Hence, $w$ can be recovered stably.
	\end{example}
	
		As a fruitage of Theorem \ref{Th3.17}, we characterize stable recovery for the source term of the non-uniform discrete dynamical system \eqref{Dynamical System} with the spectral radius $\rho(A)<1$.

	\begin{theorem}\label{Th3.19}
Suppose
\begin{enumerate}
  \item  $\{g_j\}_{j \in \Lambda} \subset \mathcal{H} $ is a Bessel sequence with  Bessel bound $\beta>0$.
  \item  $W$ is  a closed subspace of $\mathcal{H}$.
\end{enumerate}
 Then,  for the non-uniform discrete dynamical system \eqref{Dynamical System}, with $\rho(A)<1$,  each  source term $w \in W$ can be stably recovered from the measurements $D(x_{0},x_{-2},w)=[\langle x_{\lambda},g_{j}\rangle]_{\lambda,\,j \in \Lambda}$ if and only if  $\{P_{W}(I-{A^*})^{-1}g_{j}\}_{j \in \Lambda}$ is a frame for $W$.
	\end{theorem}
	
	\begin{proof}
		To prove the result, it is sufficient to prove that the conditions of Theorem \ref{Th3.17} are satisfied. Notice that the equation \eqref{Dynamical System} can be written as
		\[
		x_{\lambda}=
		\begin{cases}
			A^{\lambda}x_{0}+(I-A^{\lambda})(I-A)^{-1}w,\,\,\lambda \in  2\Z^{+};&\\
			A^{\lambda+1-\frac{r}{N}}x_{0}+(I-A^{\lambda+1-\frac{r}{N}})(I-A)^{-1}w,\,\,\lambda \in  2\Z^{+}+\frac{r}{N} \cup \left\{\frac{r}{N}\right\};&\\
			A^{-\lambda-2}x_{-2}+(I-A^{-\lambda-2})(I-A)^{-1}w,\,\,\,\lambda \in  2\Z^{-};&\\
			A^{-\lambda-1+\frac{r}{N}}x_{-2}+(I-A^{-\lambda-1+\frac{r}{N}})(I-A)^{-1}w,\,\,\lambda \in  2\Z^{-}+\frac{r}{N},
		\end{cases}
		\]
		which can be further simplified as
		\[
		x_{\lambda}-(I-A)^{-1}w=
		\begin{cases}
			A^{\lambda}\left(x_{0}-(I-A)^{-1}w\right),\,\,\lambda \in  2\Z^{+};&\\
			A^{\lambda+1-\frac{r}{N}}\left(x_{0}-(I-A)^{-1}w\right),\,\,\lambda \in  2\Z^{+}+\frac{r}{N} \cup \left\{\frac{r}{N}\right\};&\\
			A^{-\lambda-2}\left(x_{-2}-(I-A)^{-1}w\right),\,\,\lambda \in  2\Z^{-};&\\
			A^{-\lambda-1+\frac{r}{N}}\left(x_{-2}-(I-A)^{-1}w\right),\,\,\lambda \in  2\Z^{-}+\frac{r}{N}.
		\end{cases}
		\]
		Hence,
		\[
		\norm{x_{\lambda}-(I-A)^{-1}w}_{\mathcal{H}} \leq
		\begin{cases}
			\norm{A^{\lambda}}_{\mathcal{H} \rightarrow \mathcal{H}}\norm{x_{0}-(I-A)^{-1}w}_{\mathcal{H}},\,\,\lambda \in  2\Z^{+};&\\
			\norm{A^{\lambda+1-\frac{r}{N}}}_{\mathcal{H} \rightarrow \mathcal{H}}\norm{x_{0}-(I-A)^{-1}w}_{\mathcal{H}},\,\,\lambda \in  2\Z^{+}+\frac{r}{N} \cup \left\{\frac{r}{N}\right\};&\\
			\norm{A^{-\lambda-2}}_{\mathcal{H} \rightarrow \mathcal{H}}\norm{x_{-2}-(I-A)^{-1}w}_{\mathcal{H}},\,\, \lambda \in  2\Z^{-};&\\
			\norm{A^{-\lambda-1+\frac{r}{N}}}_{\mathcal{H} \rightarrow \mathcal{H}}\norm{x_{-2}-(I-A)^{-1}w}_{\mathcal{H}},\,\,\lambda \in  2\Z^{-}+\frac{r}{N}.
		\end{cases}
		\]
		By the assumption that $\rho(A)<1$, we have
		\begin{align*}
			 & \norm{A^{\lambda}}_{\mathcal{H} \rightarrow \mathcal{H}},	\norm{A^{\lambda+1-\frac{r}{N}}}_{\mathcal{H} \rightarrow \mathcal{H}} \longrightarrow 0 \; \text{as} \; \lambda \rightarrow \infty, \quad \text{and} \quad \\
				&	\norm{A^{-\lambda-2}}_{\mathcal{H} \rightarrow \mathcal{H}} , \norm{A^{-\lambda-1+\frac{r}{N}}}_{\mathcal{H} \rightarrow \mathcal{H}} \longrightarrow 0 \; \text{as} \; \lambda \rightarrow -\infty.
		\end{align*}	
		This implies,
		$	\norm{x_{\lambda}-(I-A)^{-1}w}_{\mathcal{H}} \longrightarrow 0 \; \text{as} \; \abs{\lambda} \rightarrow \infty $.	
		This suggests that the pair of the stationary points of the non-uniform discrete dynamical system is uniquely determined by $S(w)=(I-A)^{-1}w$.
		In fact, if we take $x_{0}=x_{-2}=(I-A)^{-1}w$, then $x_{\lambda}=\frac{x_{0}+x_{-2}}{2} \; \text{for all} \; \lambda \in \Lambda$.
	Also,
		\begin{align*}
			\norm{(I-A)w} = \norm{Iw-Aw}
			 \geq  \abs{\norm{w}-\norm{Aw}}
			 \geq  \norm{w}.
		\end{align*}
Thus,  $(I-A)^{-1}$ exists and bounded. Therefore, $S:=(I-A)^{-1}\restriction_W$ is bounded and invertible. Also, the adjoint operator of $S$ is given by $S^{*} = P_{W}(I-A^{*})^{-1}$. Hence, the conclusion follows from Theorem \ref{Th3.17}.
	\end{proof}
	
The following example validates  Theorem \ref{Th3.19}.

          \begin{example}
	       Let $\mathcal{H}=\ell^2(\Lambda)$ and $W=span\{e_{0},e_{\frac{r}{N}}\}$. Let $\{g_j\}_{j \in \Lambda}=\{e_{j}\}_{j \in \Lambda}$ be the standard orthonormal basis for $\ell^2(\Lambda)$ which form a Bessel sequence with optimal Bessel bound $1$. Consider the non-uniform discrete dynamical system \eqref{Dynamical System} with $A(e_{j})=\frac{1}{4}e_{j} \; \text{for all} \; j \in \Lambda$. Then $\rho(A) < 1$.
	       Let $x_{0}=x_{-2}=(I-A)^{-1}w$.
Define $S:W \longrightarrow \mathcal{H}$ by
	       $S(w)=(I-A)^{-1}w=\frac{4}{3}w$.
	       Also,
\begin{align*}
P_{W}(I-A^{*})^{-1}(g_{j})=
	       \begin{cases}
	       	\frac{4}{3}e_{j}, & \; \text{if} \; j=0, \frac{r}{N}; \\
	       	0, & \text{otherwise}.
	       \end{cases}
\end{align*}
	       Thus, $\{P_{W}(I-A^{*})^{-1}(g_{j})\}_{j \in \Lambda}=\{\cdots,0,0,\frac{4}{3}e_{0},\frac{4}{3}e_{\frac{r}{N}},0,0, \cdots\}$ is a frame for $W$.

	       Consider the canonical dual $\{\tilde{g_{j}}\}_{j \in \Lambda}=\{\cdots,0,0,\frac{3}{4}e_{0},\frac{3}{4}e_{\frac{r}{N}},0,0, \cdots\}$ of $\{P_{W}(I-A^{*})^{-1}(g_{j})\}_{j \in \Lambda}$.
	       Define $\mathscr{R}:\mathcal{B}^{s}(\ell^2(\Lambda),\ell^\infty(\Lambda)) \longrightarrow \mathcal{H}$ by $\mathscr{R}(T)=\lim{T\{\tilde{g_{j}}\}_{j \in \Lambda}}$.
	       Then $\mathscr{R}$ is a bounded linear operator such that
	       \begin{align*}
	       	\mathscr{R}(D)=&\lim{D\{\tilde{g_{j}}\}_{j \in \Lambda}}\\
	       	=& \lim_{\abs{\lambda} \rightarrow \infty}\sum_{j \in \Lambda}\langle x_{\lambda},g_{j}\rangle \tilde{g_{j}}\\
	       	=& \lim_{\abs{\lambda} \rightarrow \infty}\sum_{j=0}^{\frac{r}{N}}\langle x_{\lambda},e_{j}\rangle \frac{3}{4}e_{j}\\
	       	=& \sum_{j=0}^{\frac{r}{N}}\langle \lim_{\abs{\lambda} \rightarrow \infty}x_{\lambda},e_{j}\rangle \frac{3}{4}e_{j}\\
	       	=& \sum_{j=0}^{\frac{r}{N}}\langle S(w),e_{j}\rangle \frac{3}{4}e_{j}\\
	       	=& \sum_{j=0}^{\frac{r}{N}}\langle \frac{4}{3}w,e_{j}\rangle \frac{3}{4}e_{j}\\
	       	=& \sum_{j=0}^{\frac{r}{N}}\langle w,e_{j}\rangle e_{j}\\
	       	=&w.
	       \end{align*}
	       Hence, $w$ can be recovered stably.
         \end{example}

$$\text{\textbf{Acknowledgement}}$$
 The  research of Ruchi is   supported by the University Grants Commission  under the Grant No.~21610054096.

\textbf{Ruchi}, Department of Mathematics,
		University of Delhi, Delhi-110007, India.\\
Email: rgarg@maths.du.ac.in\\

\textbf{Lalit Kumar Vashisht}, Department of Mathematics,
		University of Delhi, Delhi-110007, India.\\
Email: lalitkvashisht@gmail.com
\end{document}